\documentclass[11pt,reqno]{amsart}
\usepackage[colorlinks=true,
pdfstartview=FitV, linkcolor=cyan, citecolor=magenta,
urlcolor=blue]{hyperref}
\usepackage{amsmath,amsfonts,latexsym,amssymb}
\usepackage{mathrsfs}
\usepackage[latin1]{inputenc}
\usepackage[T1]{fontenc}
\usepackage{ae,aecompl}
\usepackage{braket}
\usepackage{comment}
\usepackage{color}
\usepackage{graphicx}
\usepackage{subfig}
\newtheorem{theorem}{Theorem}[section]
\newtheorem{lemma}[theorem]{Lemma}
\newtheorem{proposition}[theorem]{Proposition}
\newtheorem{corollary}[theorem]{Corollary}

\newtheorem*{thm*}{\protect\theoremname}
\theoremstyle{plain}
\newtheorem*{cor*}{\protect\corollaryname}
\renewcommand{\leq}{\leqslant}
\renewcommand{\geq}{\geqslant}
\long\def\@savemarbox#1#2{\global\setbox#1\vtop{\hsize\marginparwidth 
  \@parboxrestore\tiny\raggedright #2}}
\def\eproof{$\Box$ \medskip}

\title[Pressure metrics for deformation spaces of quasifuchsian groups]
{Pressure metrics for deformation spaces of quasifuchsian groups with parabolics}
\author[Bray]{Harrison Bray}
\address{George Mason University, Fairfax, VA 22030 }
\author[Canary]{Richard Canary}
\address{University of Michigan, Ann Arbor, MI 41809}
\author[Kao]{Lien-Yung Kao}
\address{George Washington University, Washington, D.C. 20052}
\thanks{
Canary was partially supported by the grants DMS-1564362 and DMS-1906441 from the National Science Foundation (NSF). 
Kao was partially supported by the grant DMS-1703554 from NSF
}

\makeatother

\providecommand{\corollaryname}{Corollary}
\providecommand{\theoremname}{Theorem}

\begin{document}
\begin{abstract}
In this paper, we produce a  mapping class group invariant pressure metric on the space $QF(S)$ of 
quasiconformal deformations of a co-finite area Fuchsian group uniformizing $S$.
Our pressure metric arises from an analytic pressure form on $QF(S)$
which is degenerate only on pure bending vectors on the Fuchsian locus. Our techniques also
show that the Hausdorff dimension of the limit set varies analytically.
\end{abstract}

\maketitle
\maketitle

\setcounter{tocdepth}{1}
\tableofcontents{}

\section{Introduction}

We construct a pressure metric on the quasifuchsian space  $QF(S)$ of quasiconformal deformations, within
$\mathsf{PSL}(2,\mathbb C)$, of a Fuchsian group $\Gamma$ in $\mathsf{PSL}(2,\mathbb R)$ whose quotient
$\mathbb H^2/\Gamma$ has finite area and is homeomorphic to the interior of a compact surface $S$.
Our pressure metric is a mapping class group invariant path metric,
which is a Riemannian metric on the complement of the submanifold of Fuchsian representations. Our metric
and its construction generalize work of Bridgeman \cite{bridgeman-wp} in the case that $\mathbb H^2/\Gamma$
is a closed surface.

McMullen  \cite{mcmullen-pressure} initiated the study of pressure metrics, by constructing a pressure metric on
the Teichm\"uller space of a closed surface. His pressure metric is one way of formalizing Thurston's notion
of constructing a metric on Teichm\"uller space as the ``Hessian of the length of a random geodesic''
(see  also Wolpert \cite{wolpert}, Bonahon \cite{bonahon} and Fathi-Flaminio \cite{fathi-flaminio}) and like Thurston's metric 
it agrees with the classical Weil-Petersson metric (up to scalar multiplication).
Subsequently, Bridgeman \cite{bridgeman-wp} constructed a pressure metric on quasifuchsian space,
Bridgeman, Canary, Labourie and Sambarino \cite{BCLS} constructed pressure metrics on deformation spaces of Anosov
representations, and Pollicott and Sharp \cite{pollicott-sharp} constructed pressure metrics on spaces of
metric graphs (see also Kao  \cite{kao-graphs}).  The main tool in the construction of these pressure metrics is
the Thermodynamic Formalism
for topologically transitive, Anosov flows  with compact support and their associated well-behaved finite Markov codings.

The major obstruction to extending the constructions of pressure metrics to deformation spaces
of geometrically finite (rather than convex cocompact) Kleinian groups and related
settings is that the support of the recurrent portion of the geodesic flow is not compact and
hence there is not a well-behaved finite Markov coding. Mauldin-Urbanski \cite{MU} and Sarig \cite{sarig-2009} extended
the Thermodynamical Formalism to the setting of topologically mixing Markov shifts with countable alphabet and the (BIP)
property. In the case of finite area hyperbolic surfaces, Stadlbauer \cite{stadlbauer} and Ledrappier and Sarig \cite{LS}
construct and study a topologically mixing countable Markov coding with the (BIP) property for the recurrent portion of the geodesic
flow of the surface. In previous work, Kao \cite{kao-pm} showed how to adapt
the Thermodynamic Formalism in the setting of the Stadlbauer-Ledrappier-Sarig coding to 
construct pressure metrics on Teichm\"uller spaces of punctured surfaces.  

We adapt the techniques developed by Bridgeman \cite{bridgeman-wp} and Kao \cite{kao-pm} into our setting
to construct a pressure metric which can again be naturally interpreted as the Hessian of the  (renormalized) length of a random geodesic.

\begin{thm*}
[Theorem \ref{main theorem}] If $S$ is a compact surface
with non-empty boundary, the pressure form $\mathbb P$ on $QF(S)$ induces a ${\rm Mod}(S)$-invariant
path metric, which is an analytic Riemannian metric on the complement of the Fuchsian locus.

Moreover, if $v\in T_\rho(QF(S))$, then $\mathbb{P}(v,v)=0$ if and
only if  $\rho$ is Fuchsian and $v$ is a pure bending vector.
\end{thm*}

The control obtained from the Thermodynamic Formalism allows us to
see that the topological entropy of the geodesic flow  of the  quasifuchsian hyperbolic
3-manifold varies analytically over $QF(S)$.  We recall that the topological entropy $h(\rho)$ of $\rho$ is the
exponential growth rate of the number of closed orbits of the geodesic flow of $N_\rho=\mathbb H^3/\rho(\Gamma)$ of  length at most $T$.
More precisely, if 
$$R_T(\rho)=\{[\gamma]\in[\Gamma]\ |\ 0< \ell_{\rho}(\gamma)\le T\},$$
where $[\Gamma]$ is the collection of conjugacy classes in $\Gamma$ and $\ell_\rho(\gamma)$ is the
translation length of the action of $\rho(\gamma)$ on $\mathbb H^3$,  then the topological entropy
is given by
$$h(\rho)=\lim_{T\to\infty} \frac{\#R_T(\rho)}{T}.$$

Sullivan \cite{sullivan-GF} showed that the topological entropy and the Hausdorff dimension of the limit
set agree for quasifuchsian groups. So we see that the Hausdorff dimension
of the limit set varies analytically over $QF(S)$, generalizing a result of Ruelle
\cite{ruelle-hd} for quasifuchsian deformation spaces of closed surfaces.
Schapira and Tapie \cite[Thm. 6.2]{schapira-tapie} previously established that the entropy is
$C^1$ on $QF(S)$ and computed its derivative (as a special case of a much more general result).

\begin{cor*}
[Corollary \ref{h-dim analytic}] 
If $S$ is a compact surface with non-empty boundary, then the Hausdorff dimension of
the limit set varies analytically over $QF(S)$.
\end{cor*}

Concretely, the pressure form $\mathbb P$ at a representation
$\rho_0$ is the Hessian of the renormalized pressure intersection  $J(\rho_0,\cdot)$ at $\rho_0$.
The pressure intersection of $\rho,\eta\in QF(S)$ is given by
$$I(\rho,\eta)=\lim_{T\to\infty} \frac{1}{|R_T(\rho)|}\sum_{[\gamma]\in R_T(\rho)}
\frac{\ell_\eta(\gamma)}{\ell_{\rho}(\gamma)}$$
and the renormalized pressure intersection is given by
$$J(\rho,\eta)=\frac{h(\eta)}{h(\rho)}\lim_{T\to\infty} \frac{1}{|R_T(\rho)|}\sum_{[\gamma]\in R_T(\rho)}
\frac{\ell_\eta(\gamma)}{\ell_{\rho}(\gamma)}.$$
The pressure intersection was first defined by Burger \cite{burger} for pairs of convex cocompact Fuchsian
representations.
Schapira and Tapie \cite{schapira-tapie} defined an intersection function  for negatively curved manifolds with an entropy gap at infinity, by generalizing the geodesic
stretch considered by Knieper \cite{knieper} in the compact setting.  Their definition applies in a much more general framework, but 
agrees with our notion in this setting, see \cite[Prop. 2.17]{schapira-tapie}. 

Let $(\Sigma^+,\sigma)$ be the Stadlbauer-Ledrapprier-Sarig coding of a Fuchsian group $\Gamma$ giving a finite area uniformization of $S$.
If $\rho\in QF(S)$ we construct a roof function $\tau_\rho:\Sigma^+\to \mathbb R$ whose periods are translation lengths of elements of $\rho(\Gamma)$.
The key technical work in the paper is a careful analysis of these roof functions. 
In particular, we show that they vary analytically over $QF(S)$, see Proposition \ref{roof analytic}.
If $P$ is the Gurevich pressure function (on the space of all well-behaved roof functions), then the
topological entropy $h(\rho)$ of $\rho$ is the unique solution of $P(-t\tau_\rho)=0$. Our actual working definition of the intersection function
will be expressed in terms of equilibrium states on $\Sigma^+$ for the functions $-h(\rho)\tau_\rho$, but we will show in Theorem \ref{geometricintersection}
that this thermodynamical definition agrees with the more geometric definition given above.

Following Burger \cite{burger}, if $\rho,\eta\in QF(S)$, we define, the {\em Manhattan curve}
$$\mathcal C(\rho,\eta)=\{ (a,b) \ | a,b\ge 0,\ a+b>0,\ \text{ and }\ P(-a\tau_\rho-b\tau_\eta)=0\}.$$
The following result generalizes work of Burger \cite{burger} and Kao \cite{kao-manhattan}.

\begin{thm*}[Theorems \ref{Manhattan} and  \ref{geometricintersection}]
If $S$ is a compact surface with non-empty boundary,
and $\rho,\eta\in QF(S)$, then $\mathcal C(\rho,\eta)$ 
\begin{enumerate}
\item
is a closed subsegment of an analytic curve,
\item
has endpoints  $(h(\rho),0)$ and $(0,h(\eta))$,
\item
and is strictly convex, unless $\rho$ and $\eta$ are conjugate in ${\rm Isom}(\mathbb H^3)$.
\end{enumerate}
Moreover,  the tangent line to $\mathcal C(\rho,\eta)$ at $(h(\rho),0)$ has slope
$-I(\rho,\eta).$
\end{thm*}

We use Theorem \ref{Manhattan} in our proof of a rigidity result  for the renormalized pressure intersection, see Corollary 
\ref{intersection rigidity} , and in our proof that pressure intersection is analytic on $QF(S)\times QF(S)$, see
Proposition \ref{intersectionanalytic}. We also use it to obtain a rigidity theorem for weighted entropy in the spirit
of the Bishop-Steger rigidity
theorem for Fuchsian groups, see \cite{bishop-steger}. If $a,b>0$ and $\rho,\eta\in QF(S)$, we define the weighted entropy
$$h^{a,b}(\rho,\eta)=\lim \frac{1}{T} \#\{[\gamma]\in[\Gamma]\ |\ a\ell(\rho(\gamma))+b\ell(\eta(\gamma))\le T\}.$$

\medskip\noindent
\begin{cor*}[Corollary \ref{dualrigidity}]
If $S$ is a compact surface with non-empty boundary, $\rho,\eta\in QF(S)$ and $a,b>0$, then
$$h^{a,b}(\rho,\eta)\le \frac{h(\rho)h(\eta)}{bh(\rho)+ah(\eta)}$$
with equality if and only if $\rho=\eta$.
\end{cor*}

\bigskip\noindent
{\bf Other viewpoints:} If $\rho\in QF(S)$, then $N_\rho=\mathbb H^3/\rho(\Gamma)$ is a geometrically finite
hyperbolic \hbox{3-manifold.}  As such its dynamics may be analyzed using  techniques from dynamics which do not
rely on symbolic dynamics. For example, it naturally fits into the frameworks for geometrically finite
negatively curved manifolds developed by Dal'bo-Otal-Peign\'e \cite{DOP}, negatively curved Riemannian manifolds with
bounded geometry as studied by Paulin-Pollicott-Schapira \cite{PPS} and negatively curved manifolds with an entropy
gap at infinity as studied by Schapira-Tapie \cite{schapira-tapie}. 
In particular, the existence of equilibrium states and their continuous variation 
in our setting also follows from the work of  Schapira and Tapie \cite{schapira-tapie}.

Since all the geodesic flows of manifolds in $QF(S)$ are H\"older orbit equivalent, one should be able to think of
them all as arising from an analytically varying family of H\"older potential functions on the geodesic flow of a fixed hyperbolic 
\hbox{3-manifold.} However, for the construction
of the pressure metric it will be necessary to know that the pressure function is at least twice differentiable.
Results of this form do not yet seem to be available without 
symbolic dynamics. We have therefore chosen to develop the theory entirely from the viewpoint of the coding throughout the paper.

Iommi, Riquelme and Velozo \cite{IRV} have previously used the Dal'bo-Peign\'e coding \cite{dalbo-peigne} to study negatively curved manifolds of
extended Schottky type. These manifolds include the hyperbolic 3-manifolds associated to all  quasiconformal deformations of finitely generated Fuchsian
groups whose quotients have infinite area. In particular, they perform a phase transition analysis and show the existence and uniqueness of equilibrium
states in their setting. The symbolic approach to phase transition analysis can be traced back to Iommi-Jordan \cite{iommi-jordan}.
Riquelme and Velozo \cite{riquelme-velozo}
work in a more general setting which includes quasifuchsian groups with parabolics,
but without a coding, and obtain a phase transition analysis for the pressure function as well as the existence of equilibrium measures.

\medskip\noindent
{\bf Acknowledements:} The authors would like to thank Francois Ledrappier, Mark Pollicott, 
Ralf Spatzier,  and Dan Thompson for helpful conversations during the course of their investigation. We also
thank the referees  whose suggestions greatly improved the exposition.

\section{Background}
\subsection{Quasifuchsian space}
Let $S$ be a compact orientable surface with non-empty boundary and suppose that 
$\Gamma\subset\mathsf{PSL}(2,\mathbb R)$ is a discrete torsion-free group so that 
$\mathbb H^2/\Gamma$ is a finite area hyperbolic surface homeomorphic to the interior of $S$.
We say that $\rho:\Gamma\to \mathsf{PSL}(2,\mathbb C)$ is {\em quasifuchsian} if 
there exists a quasiconformal homeomorphism $\phi:\widehat{\mathbb C}\to\widehat{\mathbb C}$ such that
$\rho(\gamma)=\phi\gamma\phi^{-1}$ for all $\gamma\in\Gamma$. Equivalently, $\rho$ is quasifuchsian if
and only if there is an orientation-preserving bilipschitz homeomorphism from
$N_\rho=\mathbb H^3/\rho(\Gamma)$ to $N=\mathbb H^3/\Gamma$ in the homotopy class determined by $\rho$
(see Douady-Earle \cite{douady-earle}).
Let  $QC(\Gamma)\subset {\rm Hom}(\Gamma, \mathsf{PSL}(2,\mathbb C))$  denote the space of all  quasifuchsian representations.
We recall, see Maskit \cite[Thm. 2]{maskit-slice}, that $\rho:\Gamma\to \mathsf{PSL}(2,\mathbb C)$ is quasifuchsian if and only if
$\rho$ is discrete and faithful, $\rho(\partial S)$ is parabolic and $\rho(\Gamma)$ preserves a Jordan curve in $\widehat{\mathbb C}$.

The {\em quasifuchsian space} is given by 
$$QF(S)=QC(\Gamma)/\mathsf{PSL}(2,\mathbb C)\subset X(S)=
{\rm Hom}_{tp}(\Gamma,\mathsf{PSL}(2,\mathbb C))//\mathsf{PSL}(2,\mathbb C)$$
where ${\rm Hom}_{tp}(\Gamma,\mathsf{PSL}(2,\mathbb C))$ is the space of type-preserving representations 
of $\Gamma$ into $\mathsf{PSL}(2,\mathbb C)$ (i.e. representations taking
parabolic elements of $\Gamma$ to parabolic elements of $\mathsf{PSL}(2,\mathbb C)$).
We call $X(S)$ the relative character variety and it has the structure of a projective variety.
The space $QF(S)$ is a smooth open subset of $X(S)$, so is naturally a complex analytic manifold.
(See Kapovich \cite[Section 4.3]{kapovich-book} for details.)
Bers \cite{bers-SU} showed that $QF(S)$ admits a natural identification with $\mathcal T(S)\times\mathcal T(S)$, where $\mathcal T(S)$ is the
Teichm\"uller space of $S$.

If $\rho\in QC(\Gamma)$ and $\phi$ is a quasiconformal map such that
$\rho(\gamma)=\phi\gamma\phi^{-1}$ for all $\gamma\in\Gamma$, then
$\phi$ restricts to a $\rho$-equivariant map
\hbox{$\xi_\rho:\Lambda(\Gamma)\to\Lambda(\rho(\Gamma))$}
where $\Lambda(\rho(\Gamma))$ is the limit set of $\rho(\Gamma)$, i.e. the smallest closed $\rho(\Gamma)$-invariant subset of 
$\widehat{\mathbb C}$.
Notice that since $\xi_\rho$ is $\rho$-equivariant, it must take the attracting fixed point $\gamma^+$ of a hyperbolic
element $\gamma\in\Gamma$ to the attracting fixed point $\rho(\gamma)^+$ of $\rho(\gamma)$. Since attracting
fixed points of hyperbolic elements are dense in $\Lambda(\Gamma)$, $\xi_\rho$ depends only on $\rho$ (and not
on the choice of quasiconformally conjugating map $\phi$). 
We now record well-known fundamental properties of this limit map.

\begin{lemma} 
\label{limit map}
If $\rho\in QC(\Gamma)$, then there exists a $\rho$-equivariant bi-H\"older continuous map
$$\xi_\rho:\Lambda(\Gamma)\to\Lambda(\rho(\Gamma)).$$
Moreover, if $x\in\Lambda(\Gamma)$, then $\xi_\rho(x)$ varies complex analytically over
$QC(\Gamma)$.
\end{lemma}

\begin{proof}
Since each $\xi_\rho$ is the restriction of a quasiconformal map $\phi:\widehat{\mathbb C}\to\widehat{\mathbb C}$
and quasiconformal maps are bi-H\"older (see \cite[Thm. 10.3.2]{astala-iwaniec-martin}), $\xi_\rho$ is also
bi-H\"older.

Suppose that $\{\rho_z\}_{z\in\Delta}$ is a complex analytic family of representations in $QC(\Gamma)$ parameterized by the unit disk $\Delta$.
Sullivan \cite[Thm. 1]{sullivan2} showed that there is a continuous map $F:\Lambda(\Gamma)\times\Delta\to \widehat{\mathbb C}$, so
that if $z\in\Delta$, then $F(\cdot,z)=\xi_{\rho_z}$ and if $x\in\Lambda(\Gamma)$, then $F(x,\cdot)$ varies holomorphically in $z$. Hartogs' Theorem
then implies that $\xi_\rho(x)$ varies complex analytically over all of $QC(\Gamma)$. 
\end{proof}

\subsection{Countable Markov Shifts}
A  two-sided {\em countable Markov shift} with countable alphabet $\mathcal A$ and transition matrix 
$\mathbb T\in\{0,1\}^{\mathcal A\times\mathcal A}$ is 
the set
$$\Sigma=\{x=(x_i)\in\mathcal A^{\mathbb Z}\ |\ t_{x_ix_{i+1}}=1\ {\rm for}\ {\rm all}\ i\in\mathbb Z\}$$ 
equipped with a shift map $\sigma:\Sigma\to\Sigma$ which takes $(x_i)_{i\in\mathbb Z}$ to
$(x_{i+1})_{i\in\mathbb Z}$. Notice that the shift simply moves the letter in place $i$ into
place $i-1$, i.e. it shifts every letter one place to the left.

Associated to any two-sided countable Markov shift $\Sigma$ is the  one-sided countable Markov shift
$$\Sigma^+=\{x=(x_i)\in\mathcal A^{\mathbb N}\ |\ t_{x_ix_{i+1}}=1\ {\rm for}\ {\rm all}\ i\in\mathbb N\}$$ 
equipped with a shift map $\sigma:\Sigma^+\to\Sigma^+$ which takes $(x_i)_{i\in\mathbb N}$ to
$(x_{i+1})_{i\in\mathbb N}$. In this case, the shift deletes the letter $x_1$ and moves every other letter
one place to the left. There is a natural projection map $p^+:\Sigma\to\Sigma^+$ given by
$p^+(x)=x^+=(x_i)_{i\in\mathbb N}$ which simply forgets all the terms to the left of $x_1$. 
Notice
that $p^+\circ\sigma=\sigma\circ p^+$. We will work entirely with one-sided shifts, except in the
final section.

One says that $(\Sigma^+,\sigma)$ is {\em topologically mixing} if for all $a,b\in \mathcal A$, there exists $N=N(a,b)$
so that if $n\ge N$, then there exists $x\in\Sigma$ so that $x_1=a$ and $x_n=b$. The shift $(\Sigma^+,\sigma)$ has the
big images and pre-images property (BIP) if there exists a finite subset $\mathcal B\subset\mathcal A$ so
that if $a\in\mathcal A$, then there exists $b_0,b_1\in\mathcal B$ so that $t_{b_0,a}=1=t_{a,b_1}$.

Given a one-sided countable Markov shift $(\Sigma^+,\sigma)$ and a function  $g:\Sigma^+\to \mathbb R$, let
$$V_n(g)=\sup \{ |g(x)-g(y)|\ | \ x,y\in\Sigma^+,\  x_i=y_i \ {\rm for}\ {\rm all}\ 1\le i\le n\}$$
be the {\em $n^{\rm th}$ variation} of $g$.
We say that $g$ is {\em locally H\"older continuous} if there exists $C>0$ and $\theta\in (0,1)$
so that 
$$V_n(g)\le C\theta^n$$
for all $n\in\mathbb N$. 
We say that two locally H\"older continuous functions $f:\Sigma^+\to\mathbb R$ and $g:\Sigma^+\to\mathbb R$ are
{\em cohomologous} if there exists a locally H\"older continuous function $h:\Sigma^+\to \mathbb R$ so
that 
$$f-g=h- h\circ\sigma.$$

Sarig  \cite{sarig99} considers the associated {\em Gurevich pressure} of a 
locally H\"older continuous function $g:\Sigma^+\to\mathbb R$, given by
$$P(g)=\lim_{n\to\infty}\frac{1}{n}\log \sum_{x\in \mathrm{Fix}^n\ |\ x_1=a} e^{S_ng(x)}$$
for some (any)  $a\in\mathcal A$ where 
$$S_n(g)(x)=\sum_{i=1}^{n} g(\sigma^i(x))$$
is the {\em ergodic sum}
and $\mathrm{Fix}^n=\{x\in\Sigma^+\ |\ \sigma^n(x)=x\}$.
The pressure of a locally H\"older continuous function $f$ need not be finite, but Mauldin and Urbanski \cite{MU}
provide the following characterization of when $P(f)$ is finite.

\begin{theorem}{\rm (Mauldin-Urbanski \cite[Thm. 2.1.9]{MU})}
\label{finite pressure}
Suppose that  $(\Sigma^+,\sigma)$ is a one-sided countable Markov shift which has BIP and is topologically mixing.
If $f$ is locally H\"older continuous,  then  $P(f)$ is finite if and only if
$$Z_1(f)=\sum_{a\in\mathcal A}e^{\sup\{ f(x)\ :\ x_1=a\}}<+\infty.$$
\end{theorem}

A  Borel probability measure $m$ on $\Sigma^+$ is said to be a {\em Gibbs state} for 
a locally H\"older continuous function $g:\Sigma^+\to\mathbb R$ if there exists a constant $B>1$  
and $C\in\mathbb R$ so that
$$\frac{1}{B}\le\frac{m([a_1,\ldots,a_n])}{e^{S_ng(x)-nC}}\le B$$ 
for all $x\in [a_1,\ldots,a_n]\}$, where $[a_1,\ldots,a_n]$ is the {\em cylinder} consisting of all $x\in\Sigma^+$ so
that $x_i=a_i$ for all $1\le i\le n$. 
Sarig \cite[Thm 4.9]{sarig-2009}  shows that a locally H\"older continuous function $f$ on a topologically mixing one-sided countable Markov shift  with BIP  so that $P(f)$ is
finite admits a Gibbs state $\mu_f$.
Mauldin-Urbanski \cite[Thm 2.2.4]{MU}  show that  if a locally H\"older continuous function $f$ on a topologically mixing one-sided countable Markov shift  with BIP  
admits a Gibbs state, then  $f$ admits a unique shift invariant Gibbs state.  We summarize their work
the statement below.

\begin{theorem} {\rm (Mauldin-Urbanski \cite[Thm 2.2.4]{MU},  Sarig \cite[Thm 4.9]{sarig-2009})}
\label{uniqgibbs}
Suppose that  $(\Sigma^+,\sigma)$ is a one-sided countable Markov shift which has BIP and is topologically mixing.
If $f$ is locally H\"older continuous and $P(f)$ is finite,  then  $f$ admits a unique shift invariant Gibbs state $\mu_f$.
\end{theorem}

The {\em transfer operator} is a central tool in the Thermodynamic Formalism.
Recall that the {\em transfer operator} $\mathcal{L}_{f}\colon C^{b}(\Sigma^{+})\to C^{b}(\Sigma^{+})$
of a locally H\"older continuous function $f$ over $\Sigma^{+}$
is defined by
\[
\mathcal{L}_{f}(g)(x) ={\displaystyle \sum_{y\in\sigma^{-1}(x)}e^{f(y)}g(y)}\qquad\text{ for\ all}\ x\in\Sigma^{+}.
\]
If $(\Sigma^+,\sigma)$ is topologically mixing and has the BIP property, 
$\nu$ is a  Borel probability measure for $\Sigma^+$ and
$(\mathcal L_f)^*(\nu)=e^{P(f)}\nu$ (where $(\mathcal L_f)^*$ is the dual of transfer operator),
then $\nu$ is a Gibbs state for $f$, see Mauldin-Urbanski \cite[Theorem 2.3.3]{MU}.

A  $\sigma$-invariant Borel probability measure $m$ on $\Sigma^+$ is said to be an {\em equilibrium measure} for 
a locally H\"older continuous function $g:\Sigma^+\to\mathbb R$ if 
$$P(g)=h_\sigma(m)+\int_{\Sigma^+} g \ dm$$
where $h_\sigma(m)$ is the measure-theoretic entropy of $\sigma$ with respect to the measure $m$. 
Mauldin and Urbanski \cite{MU}  give a criterion guaranteeing the existence of a unique
equilibrum state. 

\begin{theorem} {\rm (Mauldin-Urbanski \cite[Thm. 2.2.9]{MU})}
\label{uniqeq}
Suppose that  $(\Sigma^+,\sigma)$ is a one-sided countable Markov shift which has BIP and is topologically mixing.
If $f$ is locally H\"older continuous, $\nu_f$ is a shift invariant Gibbs state for $f$ and $-\int f\ d\nu_f<+\infty$, then
$\nu_f$ is the unique equilibrium measure for $f$.
\end{theorem}

We say that $\{g_u:\Sigma^+\to \mathbb R\}_{u\in M}$ is a {\em real analytic family} if $M$ is a real analytic
manifold and for all $x\in \Sigma^+$, $u\to g_u(x)$ is a real analytic function on $M$. 
Mauldin and Urbanski \cite[Thm. 2.6.12,\ Prop. 2.6.13\ and\  2.6.14]{MU},
see also Sarig (\cite[Cor. 4]{sarig-2003},\cite[Thm 5.10\ and\  5.13]{sarig-2009}),
prove real analyticity properties of the pressure function and evaluate its derivatives.
We summarize their results in Theorem \ref{pressure analytic}.
Here the {\em variance} of a locally H\"older continuous function $f:\Sigma^+\to\mathbb R$ with respect 
to a probability measure $m$ on $\Sigma^+$ is given by
$${\rm Var}(f,m)=\lim_{n\to\infty}\frac{1}{n}\int_{\Sigma^+}S_n\Big(\big(f -\int_{\Sigma^+} f\ dm\big)^2\Big)\ dm.$$

\begin{theorem}{\rm (Mauldin-Urbanski, Sarig)}
\label{pressure analytic}
Suppose that  $(\Sigma^+,\sigma)$ is a one-sided countable Markov shift which has BIP and is topologically mixing.
If $\{g_u:\Sigma^+\to \mathbb R\}_{u\in M}$ is a real analytic family of locally H\"older continuous functions such
that $P(g_u)<\infty$ for all $u$, then $u\to P(g_u)$ is real analytic.

Moreover, if $v\in T_{u_0}M$ and there exists a neighborhood $U$ of $u_0$ in $M$ such that  \hbox{$-\int_{\Sigma^+} g_u dm_{g_{u_0}}<\infty$}
if $u\in U$, then
$$D_vP(g_u)=\int_{\Sigma^+} D_v(g_u(x))\  dm_{g_{u_0}}$$
and 
$$D_{v}^2P(g_u)={\rm Var}(D_vg_u,m_{g_u{_0}})+\int_{\Sigma^+} D_v^2g_u dm_{g_{u_0}}$$
where $m_{g_{u_0}}$ is the unique equilibrium state for $g_{u_0}$.
\end{theorem}

\subsection{The Stadlbauer-Ledrappier-Sarig coding}
Stadlbauer \cite{stadlbauer} and Ledrappier-Sarig \cite{LS} describe a one-sided countable 
Markov shift $(\Sigma^+,\sigma)$
with alphabet $\mathcal A$ which encodes the recurrent portion of the geodesic flow on $T^1(\mathbb H^2/\Gamma)$. 
In this section, we will sketch the construction of this coding and recall its crucial properties.

They begin with the classical coding of a free group, as described by Bowen and Series \cite{bowen-series}.
One begins  with a fundamental domain $D_0$ for  a free  convex cocompact Fuchsian group $\Gamma$, containing the origin  in the Poincar\'e
disk model, all of whose vertices lie in $\partial\mathbb H^2$,
so that the set $\mathcal S$ of face pairings  of $D_0$ is a minimal symmetric generating set for $\Gamma$. 
One then labels any translate $\gamma(D_0)$ by the group element $\gamma$. Any geodesic ray $r_z$  beginning at the origin
and ending at $z\in\Lambda(\Gamma)$
passes through an infinite sequence of translates, so we get a sequence
$c(z)=(\gamma_k)_{k\in\mathbb N}$. One may then turn this into an infinite sequence in $\mathcal S$ by considering
$b(z)=(\gamma_k\gamma_{k-1}^{-1})_{k\in\mathbb N}$ (where we adopt the convention that $\gamma_0=id$.) 
If $\Gamma$ is convex cocompact, this produces a well behaved one-sided Markov shift $(\Sigma_{BS}^+ ,\sigma)$ 
with finite alphabet $\mathcal S$.
The obvious map $\omega:\Sigma_{BS}^+\to\Lambda(\Gamma)$ which takes $b(z)$ to $z$
is  H\"older and $(\Sigma_{BS}^+,\sigma)$ encodes the recurrent portion of
the geodesic flow of $\mathbb H^2/\Gamma$.

If one attempts to implement this procedure when
 $\Gamma$ is not convex cocompact, then one must  omit all geodesic rays which end at a parabolic
fixed point and there is no natural way to do this from a coding perspective. Moreover, if one simply restricts $\omega$ to the allowable words
then $\omega$ will not be H\"older in this case. (To see that $\omega$ will not be H\"older, choose $x,y\in\Sigma_{BS}^+$,
so that $x_i=y_i=\alpha$ for all $1\le i\le n$, where $\alpha$ is a parabolic face-pairing,
and $x_{n+1}\ne y_{n+1}$, then $d_{\Sigma_{BS}^+}(x,y)=e^{-n}$,
while $d_{\partial\mathbb H^2}(\omega(x),\omega(y))$ is comparable to $\frac{1}{n^2}$.)

Roughly, the Stadlbauer-Ledrappier-Sarig begins with $c(z)=(\gamma_k)$ and clumps together all terms in $b(z)=(\gamma_k\gamma_{k-1}^{-1})$
which lie in a subword which
is a high power of a parabolic element. One must then append to our alphabet all powers of minimal word length parabolic elements
and and disallow infinite words beginning or ending in infinitely repeating parabolic elements. 
When $\Gamma$ is geometrically finite, but not co-finite area, Dal'bo and Peign\'e \cite{dalbo-peigne} implemented this process to
powerful effect for geometrically finite Fuchsian groups with infinite area quotients. However, when $\Gamma$ is co-finite area, 
the actual description is more intricate. The states Stadlbauer-Ledrappier-Sarig use record a finite amount of information about both the past
and the future of the trajectory.

Let $\mathcal C$ be the collection of all freely reduced words in $\mathcal S$ which have minimal word  length in their conjugacy class
and generate a maximal parabolic subgroup of $\Gamma$.
Notice that the minimal word length representative of a conjugacy class of $\alpha$ is unique up to cyclic permutation.
(One may in fact choose $D_0$ so that all but one pair of parabolic elements of $\mathcal C$ is conjugate to a face-pairing.)
Since there are only finitely many conjugacy classes of maximal parabolic subgroups of  $\Gamma$, $\mathcal C$ is finite.
They then choose a sufficiently large even number $2N$ so that the length of every element of $\mathcal C$ 
divides $2N$ and let $\mathcal C^*$ be the collection of powers of elements of $\mathcal C$ of length 
exactly $2N$.  (One may assume that two elements of $\mathcal C^*$ share a subword of length at least 2 if and only if
they are cyclic permutations of one another.)

Let $\mathcal A_1$ be the set of all strings $(b_0,b_1,\ldots,b_{2N})$ in $\mathcal S$ so that
$b_0b_1\cdots b_{2N}$ is freely reduced in $\mathcal S$ and so that
neither $b_1b_2\cdots b_{2N}$ or $b_0b_1\cdots b_{2N-1}$ lies in $\mathcal C^*$.
Let $\mathcal A_2$ be the set of all freely reduced strings  of the form $(b,w^s,w_1,\cdots, w_{k-1}, c)$
where 
$w=w_1\ldots w_{2N}\in\mathcal C^*$, $b\in\mathcal S-\{w_{2N}\}$, $1\le k\le 2N$,
$s\ge 1$ and $c\in\mathcal S-\{w_k\}$.

Let $\mathcal A=\mathcal A_1\cup\mathcal A_2$ and define functions
$$r:\mathcal A\to\mathbb N\qquad\mathrm{and}\qquad G:\mathcal A\to \Gamma$$
by letting $r(a)=1$ if $a\in\mathcal A_1$ and $r(b,w^s, w_1,\ldots ,w_{k-1}, c)=s+1$ otherwise.
If  \hbox{$a=(b_0,b_1,\ldots,b_{2N})\in\mathcal A_1$}, then $G(a)=b_1$. If \hbox{$a=(b,w^s,w_1\cdots w_{k-1}, c)$},
then let $G(a)=w^{s-1}w_1\cdots w_{k+1}$. Notice that, by construction, if $n\in\mathbb N$, then
$$\#(r^{-1}(n))\le\#(\mathcal C^*)\left(\#(\mathcal S)^{2}\right)(2N).$$
So, $r^{-1}(n)$ is always non-empty and there exists $D$ so that
$r^{-1}(n)$ has size at most $D$ for all $n\in\mathbb N$, i.e. there are at most $D$ states associated to each positive integer.

Given a geodesic ray $r_z$  beginning at the origin and ending at a point $z$ in the set $\Lambda_c(\Gamma)$ of points
in the limit set which are not parabolic fixed points, 
let $c(z)=(\gamma_k)_{k\in\mathbb N}$ be the sequence
of elements of $\Gamma$ which record the translates of $D_0$ which $r_z$ passes through. 
Let $b(z)=(b_k(z))=(\gamma_k\gamma_{k-1}^{-1})\in \mathcal S^{\mathbb N}$.
We then associate to $r_z$ a finite collection of infinite words in $\mathcal S^{\mathbb N\cup\{0\}}$, by allowing $b_0$ to be any element of
$\mathcal S$, so that $b_0b_1\cdots b_{2N}$ does not lie in $\mathcal C^*$. 

Suppose we have a word $(b_k)_{k\in\mathbb N\cup \{0\}}$ arising from the previous construction.
If \hbox{$(b_0,b_1,\ldots,b_{2N})\in\mathcal A_1$}, then let $x_1=(b_0,b_1,\ldots,b_{2N})$ and shift $(b_i)$
rightward by 1 to compute $x_2$. If not, let $x_1$ be the unique sub-string of $b_0b_1\dots b_k\cdots$ which
begins at $b_0$ and is an element of $\mathcal A_2$. Then, 
$x_1=(b_0,w^s,w_1\cdots w_{k-1},b_v)$ for some $w\in\mathcal C^*$, $s\in\mathbb N$
and $v= 2Ns+k-1$. In this case, we shift $(b_i)$ rightward by $2N(s-1)+k+1$ to compute $x_2$. 
One then simply proceeds iteratively.
By construction, if $x_i\in\mathcal A_2$, then $x_{i+1}$ must lie in $\mathcal A_1$.

\medskip\noindent
{\bf Examples:} If $\Gamma$ uniformizes a once-punctured torus, then $\mathcal S=\{\alpha,\alpha^{-1},\beta,\beta^{-1}\}$ is
a mimimal symmetric generating set for $\Gamma$ and 
$$\mathcal C=\{ \alpha\beta\alpha^{-1}\beta^{-1},\beta\alpha^{-1}\beta^{-1}\alpha,\alpha^{-1}\beta^{-1}\alpha\beta,
\beta^{-1}\alpha\beta\alpha^{-1}, \beta\alpha\beta^{-1}\alpha^{-1},  \alpha\beta^{-1}\alpha^{-1}\beta, \beta^{-1}\alpha^{-1} \beta\alpha,
\alpha^{-1} \beta\alpha\beta^{-1}\}.$$
If $\Gamma$ uniformizes a four times punctured sphere, then one may choose $D_0$ so that
$\mathcal S=\{\alpha,\alpha^{-1},\beta,\beta^{-1},\gamma,\gamma^{-1}\}$ and
$$\mathcal C=\{ \alpha,\alpha^{-1},\beta,\beta^{-1},\gamma,\gamma^{-1},\alpha\beta\gamma,\beta\gamma\alpha,\gamma\alpha\beta,
\gamma^{-1}\beta^{-1}\alpha^{-1},\beta^{-1}\alpha^{-1}\gamma^{-1},\alpha^{-1}\gamma^{-1}\beta^{-1}\}.$$

The following proposition encodes crucial properties of the coding.

\begin{proposition}
{\rm (Ledrappier-Sarig \cite[Lemma 2.1]{LS}, Stadlbauer \cite{stadlbauer})}
Suppose that $\mathbb H^2/\Gamma$ is a finite area hyperbolic surface, then
$(\Sigma^+,\sigma)$ is topologically mixing, has the big images and pre-images property (BIP), and
there exists a locally H\"older continuous  finite-to-one map 
$$\omega:\Sigma^+\to \Lambda(\Gamma)$$
so that $\omega(x)=\lim (G(x_1)\cdots G(x_n))(0)$ and
$\omega(x) = G(x_1) \omega (\sigma(x))$.
Moreover, if $\gamma$ is a hyperbolic element of $\Gamma$, then there exists $x\in\mathrm{Fix}^n$,  
for some $n\in\mathbb N$, unique up to cyclic permutation,
so that $\gamma$ is conjugate to $G(x_1)\cdots G(x_n)$.
\end{proposition}

Notice that every element of $\mathcal A$ can be preceded and succeeded by some element of $\mathcal A_1$, so
$(\Sigma^+,\sigma)$ clearly has (BIP). The topological mixing property is similarly easy to see directly from the definition,
so the main claim of this proposition is that $\omega$ is locally H\"older continuous.

Another crucial property of the coding is that the translates of the origin associated to the Stadbauer-Ledrappier-Sarig coding
approach points in the limit set conically  (see property (1) on
page 15 in Ledrappier-Sarig \cite{LS}). 

\begin{lemma}{\rm (Ledrappier-Sarig \cite[Property (1) on page 15]{LS})}
\label{conicalapproach}
Given $y\in\mathbb H^2$,
there exists $L>0$  so that if $x\in\Sigma^+$ and $n\in\mathbb N$, then
$$d(G(x_1)G(x_2)\cdots G(x_n)(0),\overrightarrow{y\omega(x)})\le L.$$
\end{lemma}

Since the proof of Lemma \ref{conicalapproach} appears in the middle of a rather technical discussion in \cite{LS}, we will
sketch a proof in our language. Choose a compact subset $\hat K$ of $\mathbb H^2/\Gamma$ so that its complement
is a collection of cusp regions bounded by curves which are images of horocycles in $\mathbb H^2$. 
Without loss of generality we may assume that $y$ is the origin in the Poincar\'e disk model for $\mathbb H^2$.
Notice that if the portion of 
$\overrightarrow{b\omega(x)}$ between $\gamma_s(D_0)$ and $\gamma_{s+t}(D_0)$ 
lies entirely in the complement of the pre-image of $\hat K$, and $t>s$, then $\gamma_{s+t}\gamma_s^{-1}$ is a subword of a 
power of an element in $\mathcal C$. Let $K$ be the intersection of the pre-image of $\hat K$ with $D_0$. 
Notice that we may assume that $y\in K$ (by perhaps enlarging $\hat K$).
Suppose the last $2N+1$ letters of $x_n$ are $b_r\cdots b_{r+2N}$, then 
$\overrightarrow{0\omega(x)}$ intersects one of $\gamma_r(K),\ \ldots, \ \gamma_{r+2N}(K)$
(since otherwise $b_r\cdots b_{r+2N-1}$ or $b_{r+1}\cdots b_{r+2N+1}$ would lie in $\mathcal C^*$, which is disallowed). But then
$$d\big(G(x_1)\cdots G(x_n)(y),\overrightarrow{y\omega(x)}\big)\le R+\mathrm{diam}(K)$$
where
$$R=\max\Big\{d(y,(s_1\ldots s_p)(y))\ |\ s_i\in\mathcal S,\ p\in\{1,\ldots,2N\}\Big\}.$$

\section{Roof functions for quasifuchsian groups}
If $\rho\in QC(\Gamma)$, we define a {\em roof function} $\tau_\rho:\Sigma^+\to \mathbb R$ by setting
$$\tau_\rho(x)=B_{\xi_\rho(\omega(x))}(b_0,\rho(G(x_1))(b_0))$$
where $b_0=(0,0,1)$ and $B_z(x,y)$ is the Busemann function based at $z\in\partial\mathbb H^3$ which measures the signed distance
between the horoballs based at $z$ through $x$ and $y$. In the Poincar\'e upper half space model, we write
the Busemann function explicitly as
$$\hat B_z(p,q)=\log\left(\frac{|p-z|^2h(p)}{|q-z|^2h(q)}\right)$$
where $z\in\mathbb C\subset\partial \mathbb H^3$, $p,q\in \mathbb H^3$ and $h(p)$ is the Euclidean height
of $p$ above the complex plane and $\hat B_\infty(p,q)=\frac{h(p)}{h(q)}$.

It follows from the cocycle property of the Busemann function that 
$$S_m\tau_\rho(x)=\sum_{i=0}^{m-1}\tau_\rho(\sigma^i(x))=B_{\xi_\rho(\omega(x))}(b_0,\rho(G(x_1)\cdots G(x_m))(b_0)).$$
In particular,  if $x=(\overline{x_1,\ldots ,x_m})\in \Sigma^+$, then
$$S_m\tau_\rho(x)=\ell_\rho(G(x_1)\cdots G(x_m)).$$
We say that the roof function  $\tau_\rho$ is {\em eventually positive} if
there exists $C>0$ and $N\in\mathbb N$ so that if $n\ge N$ and $x\in\Sigma^+$, then 
$S_n\tau_\rho(x)\ge C$.  

The following lemma records crucial properties of our roof functions. It generalizes similar results of
Ledrappier-Sarig \cite[Lemma 2.2\  and \  3.1]{LS} in the Fuchsian setting.

\begin{proposition}
\label{roof analytic}
The family 
$\{\tau_\rho\}_{\rho\in QC(\Gamma)}$ of roof functions is a real analytic family of 
locally H\"older continuous, eventually positive functions.

Moreover, if $\rho\in QC(\Gamma)$, then there exists $C_\rho>0$ and $R_\rho>0$ so that 
$$2\log r(x_1)-C_\rho\le \tau_\rho(x)\le  2\log r(x_1)+C_\rho$$
and
$$\Big|S_n\tau_\rho(x)-d(b_0,G(x_1)\cdots G(x_n))(b_0))\Big|\le R_\rho$$
for all $x\in\Sigma^+$ and $n\in\mathbb N$.
\end{proposition}

\begin{proof}
Since $\xi_\rho(q)$ varies complex analytically in $\rho$  for all $q\in\Lambda(\Gamma)$, by Lemma \ref{limit map},
and $B_z(b_0,y)$ is real analytic in  $z\in\widehat{\mathbb C}$ and $y\in\mathbb H^3$, we see that
$\tau_\rho(x)$ varies analytically over $QC(\Gamma)$ for all $x\in\Sigma^+$.

Recall, see Douady-Earle \cite{douady-earle}, that there exists $K=K(\rho)>1$ and a $\rho$-equivariant 
$K$-bilipschitz map $\phi:\mathbb H^2\to\mathbb H^3$ so that $\phi(y_0)=b_0$ where $y_0$ is the origin in the disk model for $\mathbb H^2$. 
Therefore, if $L$ is the constant
from Lemma \ref{conicalapproach} and $x\in\Sigma^+$, then $\rho(G(x_1)\cdots G(x_n))(b_0)$ lies within $KL$ of
the $K$-bilipschitz ray $\phi\Big(\overrightarrow{y_0\omega(x)}\Big)$. The Fellow Traveller property for $\mathbb H^3$ implies that
there exist $R=R(K)>0$ so that any $K$-bilipschitz geodesic ray lies a Hausdorff distance at most $R$ from the geodesic
ray with the same endpoints. Therefore, if $M=KL+R$, then, for all $n\in\mathbb N$,
$$d(\rho(G(x_1)\cdots G(x_n))(b_0),\overrightarrow{b_0\xi_\rho(\omega(x))}\le M.$$

\medskip

We next obtain our claimed bounds on the roof function.
If $x\in\Sigma^+$, then
$$|\tau_\rho(x)|\le d\big(\rho(G(x_1))(b_0),b_0\big)$$
so if $a\in \mathcal A$, there exists $C_a$ so that if $x_1=a$,
then $|\tau_\rho(x)|\le C_a.$
Since our alphabet is infinite, our work is not done.

If $w\in\mathcal C^*$, we may normalize so that $\rho(w)(z)=z+1$ and $b_0=(0,0,b_w)$ in the upper half-space model 
for $\mathbb H^3$. If $z\in\mathbb C\subset\partial\mathbb H^3$ and $r>0$, we let
$B(z,r)$ denote the Euclidean ball of radius $r$ about $z$ in $\mathbb C$.
Since $g_a$ has length at most $2N+1$ in the alphabet $\mathcal S$, we may define
$$c_w=\max\{|\rho(g_a)(b_0)|\ |\ G(a)=w^sg_a\ \mathrm{for}\ \mathrm{some}\ a\in\mathcal A_2\}$$
where $|\rho(g_a)(b_0)|$ is the Euclidean distance from $\rho(g_a)(b_0)$ to $0=(0,0,0)$.
Suppose that $x\in\Sigma^+$, $r(x_1)\ge 2$ and $G(x_1)=w^{s}g_a$  where $s=r(a)-2$.
By definition, $\rho(g_a)(b_0)\in B(0,c_w)$,  so
$$\rho(w^sg_a)(b_0)=\rho(w^s)\big(\rho(g_a)(b_0)\big)\in \rho(w^s)\big(B(0,c_w)\big)=B(s,c_w).$$
Let $S=\max\{e^Mc_w\ :\ w\in\mathcal C^*\}$.
If $s>S$, then $b_0$ does not lie in  $B(s,e^Mc_w)$, but $\overrightarrow{b_0\xi_\rho(\omega(x))}$ passes through  $B(s,e^Mc_w),$
which implies that $\xi_\rho(\omega(x))\in B(s,e^Mc_w)$.
It then follows from our formula for the Busemann function that
\begin{eqnarray*}
\tau_\rho(x) &= & \log\left(\frac{|b_0-\xi_\rho(\omega(x))|^2 h(\rho(w^sg_a)(b_0))}{|\rho(w^sg_a)(b_0)-\xi_\rho(\omega(x))|^2 h(b_0)}\right)\\
&\le &
\log\left(\frac{(b_w^2+(s+e^Mc_w)^2)h(\rho(g_a)(b_0))}{h(\rho(g_a)(b_0))^2b_w}\right)=\log\left(\frac{(b_w^2+(s+e^Mc_w)^2)}{h(\rho(g_a)(b_0))b_w}\right).\\
\end{eqnarray*}
Similarly, 
$$\tau_\rho(x)\ge \log\left(\frac{(b_w^2+(s-e^Lc_w)^2)h(\rho(g_a)(b_0))}{\big(h(\rho(g_a)(b_0))^2+e^{2M}c_w^2\big)b_w}\right).$$
Since there are only finitely many choices of $g_a$, it is easy to see that there exists  $C_w$ so that 
$$2\log(r(x_1))-C_w\le \tau_\rho(x) \le 2\log(r(x_1))+C_w$$
whenever  $x\in\Sigma^+$, $r(x_1)>S+2$ and $G(x_1)=w^{s}g_a$.
Since there are only finitely many $w$ in $\mathcal C^*$ and
only finitely many words $a$ with $r(a)\le S+2$, we see that there exists $C_\rho$ so that
$$2\log(r(x_1))-C_\rho\le \tau_\rho(x) \le 2\log (r(x_1))+C_\rho$$
for all $x\in\Sigma^+$. 

\medskip

We next show that $\tau_\rho$ is locally H\"older continuous.
Since $\omega$ is locally H\"older continuous,
there exists $A$ and $\alpha>0$ so that if $x,y\in\Sigma^+$ and $x_i=y_i$ for $1\le i\le n$, then 
$$d(\omega(x),\omega(y))\le Ae^{-\alpha n}.$$
Since $\xi_\rho$ is H\"older, there exist $C$ and $\beta>0$
so that $d(\xi_\rho(z),\xi_\rho(w))\le Cd(z,w)^\beta$ for all \hbox{$z,w\in \Lambda(\Gamma)$},
so 
$$d(\xi_\rho(\omega(x)),\xi_\rho(\omega(y))\le CA^\beta e^{-\alpha\beta n}.$$
If $a\in\mathcal A$, then let 
$$D_a=\sup\left\{\Big|\frac{\partial}{\partial z}\Big|_{z=z_0}\Big(B_{z}(b_0,\rho(G(a))(b_0)\Big)\Big|\ :\ z_0=\xi_\rho(\omega(x))\ {\rm and}\ x_1=a\right\},$$
so
$$\sup \{ |\tau_\rho(x)-\tau_\rho(y)|\ | \ x,y\in [a,x_2,\ldots,x_n]\}\le D_aCA^\beta  e^{-\alpha\beta n}.$$
However, the best general estimate one can have on  $D_a$ is $O(r(a))$, so we will have to dig a little
deeper.

We again work in the upper half-space model,  and assume that $r(a)>S+2$, $G(a)=w^{s}g_a$  where $s=r(a)-2$
and normalize as before so that $\rho(w) (z)=z+1$. We then map the limit set into the boundary of
the upper-half space model by setting $\hat\xi_\rho=T\circ\xi_\rho$ where $T$ is a conformal automorphism which takes the Poincar\'e ball model to
the upper half-space model and takes the fixed point of $\rho(w)$ to $\infty$. Notice that $T$ is $K_w$-bilipschitz
on $T^{-1}(B(0,e^Mc_w))$. Therefore, if $x,y\in [a,x_2,\ldots,x_n]$, then 
$$|\hat\xi_\rho(x)-\hat\xi_\rho(y)=|\hat\xi_\rho(w^{-s}(x))-\hat\xi_\rho(w^{-s}(x))|\le K_wCA^\beta e^{-\alpha\beta(n-1)}$$
Moreover, there exists $D_w$ so that
$$\left|\frac{\partial}{\partial z}\Big|_{z=z_0}\Big(\hat B_{z}(b_0,\rho(G(a))(b_0)\Big)\right|\le D_w$$
if $z_0\in \rho(w)^{s}(B(0,e^Mc_w)))$, so
$$\sup \{ |\tau_\rho(x)-\tau_\rho(y)|\ \big| \ x,y\in [a,x_2,\ldots,x_n]\}\le K_wD_wCA^\beta  e^{-\alpha\beta (n-1)}.$$
Since there are only finitely many $a$ where $r(a)\le S+2$ and only finitely many choices of $w$, our
bounds are uniform over $\mathcal A$ and so $\tau_\rho$ is locally H\"older continuous.

\medskip

It remains to check that $\tau_\rho$ is eventually positive. Since
$$d(\rho(\gamma_n)(b_0),\overrightarrow{b_0\xi_\rho(\omega(x))})\le M$$
for all $n\in\mathbb N$, we see that
$$\Big|S_n\tau_\rho(x)-d(b_0,G(x_1)\cdots G(x_n))(b_0))\Big|\le 2M=R_\rho$$
Since the set
$$\mathcal B=\{\gamma\in \Gamma\ |\  d(\rho(\gamma)(b_0),b_0) \le 2R_\rho \}$$
is finite, there exists $\hat N$ so that if $\gamma$ has word length at least
$\hat N$ (in the generators given $\mathcal S$), then $\gamma$ does not
lie in $\mathcal B$. Therefore, if $n\ge \hat N$ and $x\in\Sigma^+$,
then $S_n\tau_\rho(x)>R_\rho>0$. Thus, $\tau_\rho$ is eventually positive and our proof is complete.
\end{proof}

It is a standard feature of the Thermodynamic Formalism that one may replace an eventually positive
roof function by a roof function which is strictly positive and cohomologous to the original roof function.
(For a statement and proof which  includes the current situation, see \cite[Lemma 3.3]{BCKM}.)

\begin{corollary}
\label{positive roof}
If $\rho\in QC(\Gamma)$, there exists a locally H\"older continuous function $\hat\tau_\rho$ and $c>0$ so that 
$\hat\tau_\rho(x)\ge c$ for all $x\in \Sigma^+$ and $\hat\tau_\rho$ is  cohomologous to $\tau_\rho$.
\end{corollary}

\section{Phase transition analysis}

We begin by extending  Kao's  phase transition analysis, see
Kao \cite[Thm. 4.1]{kao-pm}, which characterizes which linear combinations of  a pair of roof functions have finite pressure.
The primary use of this analysis will be in the case of a single roof function, i.e. when $a=1$ and $b=0$. However, we will
use the full force of this result in the proof of our Manhattan curve theorem, see Theorem \ref{Manhattan}.

\begin{theorem}
\label{phase transition}
If $\rho,\eta\in QC(\Gamma)$, $t\in\mathbb R$ and $a+b>0$, then
\hbox{$P(-t(a\tau_\rho+b\tau_\eta))$} is finite if and only if $t>\frac{1}{2(a+b)}$.
Moreover, $P(-t(a\tau_\rho+b\tau_\eta))$ is monotone decreasing and analytic in $t$ on $(\frac{1}{2(a+b)},\infty)$, and 
$$\lim_{t\to\frac{1}{2(a+b)}^+} P(-t(a\tau_\rho+b\tau_\eta))=+\infty.$$
If, in addition $a,b\ge 0$, then
$$\lim_{t\to\infty} P(-t(a\tau_\rho+b\tau_\eta))=-\infty.$$
\end{theorem}

Riquelme and Velozo \cite[Thm. 1.4]{riquelme-velozo} previously established results closely related to 
Theorem \ref{phase transition} in the more
general setting of negatively curved manifolds with bounded geometry.

\begin{proof}
Recall, from Theorem \ref{finite pressure}, that, since $-t(a\tau_\rho+b\tau_\eta)$ is locally H\"older continuous and $(\Sigma^+,\sigma)$ is a one-sided,
toplogically mixing countable Markov shift with BIP,  $P(-t(a\tau_\rho+b\tau_\eta))$ is finite if and only if $Z_1(-t(a\tau_\rho+b\tau_\eta))<+\infty.$
Since there exists $D\in\mathbb N$ so that $\#r^{-1}(n)\le D$ for all $n\in\mathbb N$,
Proposition \ref{roof analytic} implies that
$$Z_1(-t(a\tau_\rho+b\tau_\eta))\le D \sum_{n=1}^\infty  e^{-t(a+b)(2\log n -\max\{C_\rho,C_\eta\})}$$ 
so  $P(-t(a\tau_\rho+b\tau_\eta))<+\infty$ if $t>\frac{1}{2(a+b)}$.
Similarly, since $r^{-1}(n)$ is non-empty if $n\ge1$,
we see that 
$$Z_1(-t(a\tau_\rho+b\tau_\eta))\ge \sum_{n=1}^\infty  e^{-t(a+b)(2\log n +\max\{C_\rho,C_\eta\})}$$ 
so  $P(-t(a\tau_\rho+b\tau_\eta))=+\infty$  if $t\le \frac{1}{2(a+b)}$ and
$$\lim_{t\to\frac{1}{2(a+b)}^+} Z_1(-t(a\tau_\rho+b\tau_\eta))=+\infty.$$

It follows from the definition that $P(-t(a\tau_\rho+b\tau_\eta))$ is monotone decreasing in $t$ and 
Theorem \ref{pressure analytic}
implies that it is analytic in $t$ on $(\frac{1}{2(a+b)},\infty)$.
In the proof of \cite[Thm. 2.1.9]{MU}, Mauldin and Urbanski show that 
given  a locally H\"older continuous function $f$ on a one-sided countable Markov shift which is topologically mixing and has property BIP,
there exist constants \hbox{$q,s,M,m>0$} so that for any $n\in\mathbb N$, we have 
$$
\sum_{i=n}^{n+s(n-1)}Z_{i}(f)\geq  \frac{e^{-M+(M-m)n}}{q^{n-1}}Z_{1}(f)^{n}.
$$
where if $E^n$ is the set of allowable words of length $n$ in $\mathcal A$, then
$$Z_n(f)=\sum_{w\in E^n} e^{\sup\{ S_nf(x)\ |\ x_i=w_i\ \forall 1\le i\le n\}}\ \text{ and }\ \lim\frac{1}{n}\log Z_n(f)=P(f).$$
It  follows that for all $n$, there exist $A>0$ and  $\hat n\in [n,n+s(n-1)]$ such that
$Z_{\hat n}\ge A^nZ_1(f)^n$, so $P(f)\ge \frac{1}{1+s}Z_1(f)-\log A$.
Therefore,
$$\lim_{t\to\frac{1}{2(a+b)}^+} P(-t(a\tau_\rho+b\tau_\eta))=+\infty.$$

If  $a,b\ge 0$ and $x\in \mathrm{Fix}^n$, then $S_n(a\tau_\rho+b\tau_\eta)(x)>0$, so if $t>1$, then
$$\sum_{x\in \mathrm{Fix}^n\ |\ x_1=a} e^{S_n(-t(a\tau_\rho+b\tau_\eta))(x)}
\le \frac{1}{t} \sum_{x\in \mathrm{Fix}^n\ |\ x_1=a} e^{S_n(-a\tau_\rho-b\tau_\eta)(x)}$$
since $c^t\le \frac{1}{t}c$ if $0\le c\le 1$ and $t>1$.
Therefore, $P(-t(a\tau_\rho+b\tau_\eta))\le P(-a\tau_\rho-b\tau_\eta)-\log t$,
so $\lim_{t\to\infty} P(-t(a\tau_\rho+b\tau_\eta))=-\infty$. 
\end{proof}

\section{Entropy and Hausdorff dimension}
\label{hd and entropy}

Theorem \ref{phase transition} implies that if $\rho\in QC(\Gamma)$ then there is a unique solution $h(\rho)>\frac{1}{2}$ to 
\hbox{$P(-h(\rho)\tau_\rho)=0$.}
This unique solution $h(\rho)$ is the {\em topological
entropy} of $\rho$, see the discussion in Kao \cite[Section 5]{kao-pm}. Theorem \ref{pressure analytic} and the implicit function theorem
then imply that $h(\rho)$ varies analytically over $QC(\Gamma)$, generalizing a result of
Ruelle \cite{ruelle-hd} in the convex cocompact case. Since  the entropy $h(\rho)$ is invariant under conjugation,
we obtain analyticity of entropy over $QF(S)$. We recall that Schapira and Tapie \cite[Thm. 6.2]{schapira-tapie} previously established that the entropy is
$C^1$ on $QF(S)$.

\begin{theorem}\label{entropy analytic}
If $S$ is a compact hyperbolic surface with non-empty boundary, then
the topological entropy  varies analytically over $QF(S)$.
\end{theorem}

Sullivan \cite{sullivan-GF} showed that the topological entropy $h(\rho)$ agrees with
the Hausdorff dimension of the limit set $\Lambda(\rho(\Gamma))$, so we obtain the following corollary.

\begin{theorem} {\rm (Sullivan \cite{sullivan-GF,sullivan-aspects})}
\label{entropy and pressure}
If $\rho\in QC(\Gamma)$, then its topological entropy $h(\rho)$ is the exponential growth rate 
of the number of  closed geodesics of length less than $T$ in $N_\rho=\mathbb H^3/\rho(\Gamma)$.
Moreover, $h(\rho)$ is the Hausdorff dimension  of the limit set $\Lambda(\rho(\Gamma))$ and the
critical exponent of the Poincar\'e series $Q_\rho(s)$.
\end{theorem}

Theorems \ref{entropy analytic} and \ref{entropy and pressure} together imply that the Hausdorff dimension
of the limit set varies analytically.

\begin{corollary}
\label{h-dim analytic}
The Hausdorff dimension of $\Lambda(\rho(\Gamma))$ varies analytically over $QC(\Gamma)$.
\end{corollary}

\medskip\noindent
{\bf Remarks:} 1) Sullivan \cite{sullivan-aspects} also showed that $h(\rho)$ is the critical exponent of the Poincar\'e series 
$$Q_\rho(s)=\sum_{\gamma\in\Gamma} e^{-sd(b_0,\rho(\gamma)(b_0))},$$
i.e. $Q_\rho(s)$ diverges if $s<h(\rho)$ and converges if $s>h(\rho)$.

2) Bowen \cite{bowen-qf} showed that if \hbox{$\rho\in QF(S)$} and $S$ is a closed surface,
then $h(\rho)\ge 1$ 
with equality  if and only if $\rho$ is Fuchsian. Sullivan \cite[p. 66]{sullivan-discrete}, see also Xie \cite{xie}, 
observed that Bowen's rigidity result extends to the case when $\mathbb H^2/\Gamma$ has finite area.

\section{Manhattan curves}

If $\rho,\eta\in QC(\Gamma)$, we define, following Burger \cite{burger}, the {\em Manhattan curve}
$$\mathcal C(\rho,\eta)=\{ (a,b)\in D \ | \ P(-a\tau_\rho-b\tau_\eta)=0\}$$
where $D=\{(a,b)\in\mathbb R^2\ |\ a,b\ge 0 \ {\rm and}\ (a,b)\ne (0,0)\}$.
Notice that, since the Gurevich pressure is defined in terms of lengths of closed geodesics,
if $\hat\rho$ is conjugate (or complex conjugate) to $\rho$ and $\hat\eta$ is conjugate (or complex conjugate) to $\eta$,
then
$\mathcal C(\rho,\eta)=\mathcal C(\hat\rho,\hat\eta)$.

One may give an alternative characterization by noticing that
$P(-ab_\rho-b\tau_\eta)=0$ if and only if 
$$h^{a,b}(\rho,\eta)=\lim \frac{1}{T}\log \#\{\ [\gamma]\in[\Gamma]\ |\  0<a\ell_\rho(\gamma)+b\ell_\eta(\gamma)\le T\}=1$$
where $[\Gamma]$ is the collection of conjugacy classes in $\Gamma$.
Moreover, $h^{a,b}(\rho,\eta)$ is also the critical exponent of 
$$Q_{\rho,\eta}^{a,b}(s)=\sum_{\gamma\in\Gamma} e^{-s \left(ad(0,\rho(\gamma)(0))+bd(0,\eta(\gamma)(0))\right)}.$$
(see Theorem 4.8, Remark 4.9 and Lemma 4.10 in Kao \cite{kao-manhattan}).

\begin{theorem} 
\label{Manhattan}
If $\rho,\eta\in QC(\Gamma)$, then $\mathcal C(\rho,\eta)$ 
\begin{enumerate}
\item
is a closed subsegment of an analytic curve,
\item
has endpoints  $(h(\rho),0)$ and $(0,h(\eta))$,
\item
and is strictly convex, unless $\rho$ and $\eta$ are conjugate in ${\rm Isom}(\mathbb H^3)$.
\end{enumerate}
Moreover,  the tangent line to $\mathcal C(\rho,\eta)$ at $(h(\rho),0)$ has slope
$$-\frac{\int \tau_\eta dm_{-h(\rho)\tau_\rho}}{\int \tau_\rho dm_{-h(\rho)\tau_\rho}}.$$
\end{theorem}

Burger \cite{burger} established Theorem \ref{Manhattan} for convex cocompact
Fuchsian groups, with the exception of the analyticity of the Manhattan curve, which was established by Sharp \cite{sharp}.

\medskip

Notice that if $\rho$ and $\eta$ are conjugate in ${\rm Isom}(\mathbb H^3)$, then
$\tau_\rho=\tau_\eta$ so $\mathcal C(\rho,\eta)$ is a straight line.
We will  need  the following technical result in the  proof of Theorem \ref{Manhattan}.

\begin{lemma}
\label{roof integrable}
If $\rho,\eta,\theta\in QC(\Gamma)$, $2(a+b)> 1$ and $P(-a\tau_\rho-b\tau_\eta)=0$,
then there exists a unique equlibrium state $m_{-a\tau_\rho-b\tau_\eta}$ for $-a\tau_\rho-b\tau_\eta$ and
$$0<\int_{\Sigma^+} \tau_\theta dm_{-a\tau_\rho-b\tau_\eta}<+\infty.$$
\end{lemma}

\begin{proof}
Notice that since $P(-a\tau_\rho-b\tau_\eta)=0$,
%and $\sup (-a\tau_\rho-b\tau_\eta)\le |a|C_\rho+|b|C_\eta$, 
there exists a unique shift-invariant Gibbs state $m_{-a\tau_\rho-b\tau_\eta}$ for $-a\tau_\rho-b\tau_\eta$,
see Theorem \ref{uniqgibbs}.  However, by \cite[Lemma 2.2.8]{MU},
$$\int_{\Sigma^+} a\tau_\rho+b\tau_\eta \ dm_{-a\tau_\rho-b\tau_\eta}<+\infty$$
if and only if 
$$\sum_{s\in\mathcal A} I(a\tau_\rho+b\tau_\eta,s)e^{I(-a\tau_\rho-b\tau_\eta,s)}<\infty$$
where $I(f,s)=\inf\{f(x)\ |\ x\in\Sigma,\ x_1=s\}$.
But, by Proposition \ref{roof analytic},
\begin{eqnarray*}
\sum_{a\in\mathcal A} \inf( a\tau_\rho +b\tau_\eta |_{[a]})e^{\inf ( -a\tau_\rho-b\tau_\eta |_{[a]})} & \le &
D\sum_{n\in\mathbb N}(|a|C_\rho+|b|C_\eta +2(a+b)\log n)e^{|a|C_\rho+|b|C_\eta -2(a+b)\log n} \\
&= & De^{|a|C_\rho+|b|C_\eta}\sum_{n\in \mathbb N} \frac{(|a|C_\rho+|b|C_\eta +2(a+b)\log n)}{n^{2(a+b)}}\\
\end{eqnarray*}
which converges, since $2(a+b)>1$.
Theorem \ref{uniqeq} then implies that 
$dm_{-a\tau_\rho-b\tau_\eta}$ is the unique equilibrium state for $-a\tau_\rho-b\tau_\eta$.

Proposition \ref{roof analytic}  implies that there exists $B>1$ so that  if $n$ is large enough,
then 
$$\frac{1}{B}\le \frac{\tau_\theta(x)}{a\tau_\rho(x)+b\tau_\eta(x)}\le B$$
for all $x\in\Sigma^+$ so that $r(x_1)>n$.
(For example, if $\log n>4\max\{aC_\rho+bC_\eta,C_\theta,1\}$, then we may choose $B=8(a+b).$)
Since $\tau_\theta$ is locally H\"older continuous, it is bounded on
the remainder of $\Sigma^+$. Therefore, since $\int_{\Sigma^+} a\tau_\rho+b\tau_\eta \ dm_{-a\tau_\rho-b\tau_\eta}<+\infty$, we see that
$$\int_{\Sigma^+} \tau_\theta\ dm_{-a\tau_\rho-b\tau_\eta}<+\infty.$$

Now notice that, since $\tau_\theta$ is cohomologous to a positive function $\hat\tau_\theta$, by Corollary \ref{positive roof},
$$\int_{\Sigma^+} \tau_\theta dm_{-a\tau_\rho-b\tau_\eta}=\int_{\Sigma^+} \hat\tau_\theta dm_{-a\tau_\rho-b\tau_\eta}>0.$$
\end{proof}

\medskip\noindent
{\em Proof of Theorem \ref{Manhattan}:}
Recall that $t=h(\rho)$ is the unique solution to the equation $P(-t\tau_\rho)=0$ (see the discussion at the beginning of Section \ref{hd and entropy}). So,
the intersection of the Manhattan curve with the boundary of $D$ consists of the points
$(h(\rho),0)$ and $(0,h(\eta))$.

Let 
$$\hat D= \{(a,b)\in\mathbb R^2\ | a+b>\frac{1}{2}\}.$$
Theorem \ref{phase transition} implies that  $P$ is finite on $\hat D$.
Lemma \ref{roof integrable} implies that if $a,b\in\hat D$ and \hbox{$P(-a\tau_\rho-b\tau_\eta)=0$}, then there is
an equilibrium state $m_{-a\tau_\rho-b\tau_\eta}$ for $-a\tau_\rho-b\tau_\eta$ and that
$\int_{\Sigma^+} \tau_\theta\ dm_{-a\tau_\rho-b\tau_\eta}$ is finite for all $\theta\in QC(\Gamma)$.
Theorem \ref{pressure analytic} then
implies that
$$\frac{\partial}{\partial a} P(-a\tau_\rho-b\tau_\eta)=\int_{\Sigma^+} -\tau_\rho\ dm_{-a\tau_\rho-b\tau_\eta}$$
and
$$\frac{\partial}{\partial b} P(-a\tau_\rho-b\tau_\eta)=\int_{\Sigma^+} -\tau_\eta\ dm_{-a\tau_\rho-b\tau_\eta}.$$
Since \hbox{$\int_{\Sigma^+}-\tau_\rho \ dm_{-a\tau_\rho-b\tau_\eta}$} and 
\hbox{$\int_{\Sigma^+} -\tau_\eta \ dm_{-a\tau_\rho-b\tau_\eta}$} are both non-zero,
$P$ is a submersion on $\hat D$.
Since $P$ is analytic on $\hat D$, the implicit function theorem then implies that
$$\widehat{\mathcal C}(\rho,\eta)=\{ (a,b)\in \hat D \ | \ P(-a\tau_\rho-b\tau_\eta)=0\}$$
is an analytic curve and that if $(a,b)\in\mathcal C(\rho,\eta)$ then the slope of
the tangent line to $\mathcal C(\rho,\eta)$ at $(a,b)$ is given by
$$c(a,b)=-\frac{\int_{\Sigma^+}\tau_\eta\ dm_{-a\tau_\rho-b\tau_\eta}}{\int_{\Sigma^+} \tau_\rho\ dm_{-a\tau_\rho-b\tau_\eta}}.$$

Notice that $\mathcal C(\rho,\eta)$ is the lower boundary of the
region 
$$\widehat{\mathcal C}(\rho,\eta)=\{(a,b)\ |\  Q_{\rho,\eta}^{a,b}(1)<\infty\}$$
The H\"older inequality implies that if $(a,b),(c,d)\in\widehat{\mathcal C}(\rho,\eta)$ and $t\in[0,1]$, then
$$Q_{\rho,\eta}^{ta+(1-t)c,tb+(1-t)d}\le Q(a,b)^tQ(c,d)^{1-t}$$
so $\widehat{\mathcal C}(\rho,\eta)$ is convex. Therefore, $\mathcal C(\rho,\eta)$ is convex.

A convex analytic curve is strictly convex 
if and only if it is not a line, so it remains to show that $\rho$ and $\eta$ are conjugate in ${\rm Isom}(\mathbb H^3)$ if
$\mathcal C(\rho,\eta)$ is a straight line.
So suppose that $\mathcal C(\rho,\eta)$ is a straight line with slope $c=-\frac{h(\rho)}{h(\eta)}$. 
In particular,
\begin{equation}
\label{slope value}
\frac{h(\rho)}{h(\eta)}=-c=-c(h(\rho),0)=
\frac{\int_{\Sigma^+} \tau_\eta dm_{-h(\rho)\tau_\rho}}{\int_{\Sigma^+} \tau_\rho dm_{-h(\rho)\tau_\rho}}=-c(0,h(\eta))
=\frac{\int_{\Sigma^+}  \tau_\eta dm_{-h(\eta)\tau_\eta}}{\int_{\Sigma^+}  \tau_\rho dm_{-h(\eta)\tau_\eta}}.
\end{equation}

By definition,
$$h(m_{-h(\eta)\tau_\eta})-h(\eta)\int_{\Sigma^+} \tau_\eta \ dm_{-h(\eta)\tau_\eta}=0$$
so, applying equation (\ref{slope value}), we see that 
$$h(m_{-h(\eta)\tau_\eta})-h(\rho)\int_{\Sigma^+} \tau_\rho\ dm_{-h(\eta)\tau_\eta}
=h(\eta)\int_{\Sigma^+} \tau_\eta \ dm_{-h(\eta)\tau_\eta}-h(\rho)\int_{\Sigma^+} \tau_\rho\ dm_{-h(\eta)\tau_\eta}=0.$$
Since $P(-h(\rho)\tau_\rho)=0$, this implies that $m_{-h(\eta)\tau_\eta}$ is an equilibrium measure for
$-h(\rho)\tau_\rho$. Therefore, by uniqueness of equilibrium measures we see that
$m_{-h(\eta)\tau_\eta}=m_{-h(\eta)\tau_\rho}$. Sarig  \cite[Thm. 4.8]{sarig-2009} 
showed that this only happens
when $-h(\rho)\tau_\rho$ and $-h(\eta)\tau_\eta$ are cohomologous,  so the Livsic Theorem
\cite[Thm. 1.1]{sarig-2009} (see also Mauldin-Urbanski \cite[Thm. 2.2.7]{MU}) implies that
$$\ell_\rho(\gamma)=\frac{h(\eta)}{h(\rho)}\ell_\eta(\gamma)$$
for all $\gamma\in\Gamma$. 
Kim \cite[Th, 3]{kim} proved that if $\ell_\rho(\gamma)=c\ell_\eta(\gamma)$
for all $\gamma\in\Gamma$, then $\rho$ and $\eta$ are conjugate in 
${\rm Isom}(\mathbb H^3)$. So,  we have completed the proof.
\eproof

As a nearly immediate corollary one obtains a generalization of the rigidity results of 
Bishop-Steger \cite{bishop-steger} and Burger \cite{burger}.

\begin{corollary}\label{dualrigidity}
If $\rho,\eta\in  QC(\Gamma)$ and $(a,b)\in D$, then
$$h^{a,b}(\rho,\eta)\le \frac{h(\rho)h(\eta)}{bh(\rho)+a h(\eta)}$$
with equality if and only if $\rho$ and $\eta$ are conjugate in ${\rm Isom}(\mathbb H^3)$.
\end{corollary}

\section{Pressure intersection}

We define the {\em pressure intersection} on $ QC(\Gamma)\times  QC(\Gamma)$ given by
$$I(\rho,\eta)=\frac{\int_{\Sigma^+} \tau_\eta\ dm_{-h(\rho)\tau_\rho}}{\int_{\Sigma^+} \tau_\rho\ dm_{-h(\rho)\tau_\rho}}.$$
It follows from Lemma \ref{roof integrable} that $I(\rho,\eta)$ is well-defined.
We also define a {\em renormalized pressure intersection}
$$J(\rho,\eta)=\frac{h(\eta)}{h(\rho)} I(\rho,\eta).$$

We notice that the pressure intersection and renormalized pressure intersection vary analytically in $\rho$ and $\eta$.

\begin{proposition}
\label{intersectionanalytic}
Both $I(\rho,\eta)$ and $J(\rho,\eta)$ vary analytically over $QC(\Gamma)\times QC(\Gamma)$.
\end{proposition}

\begin{proof}
Notice that, by Theorem \ref{phase transition}, Proposition \ref{roof analytic} and Theorem \ref{pressure analytic}, 
$P=P(-a\tau_\rho-b\tau_\eta)$ is analytic on 
$$R=\{(\rho,\eta,(a,b),t)\in QC(\Gamma)\times QC(\Gamma)\times\hat D\}.$$
Since we observed, in the proof of Theorem \ref{Manhattan}, that  the restriction of $P$ to $\{\rho\}\times\{\eta\}\times \hat D$ 
is a submersion for all $\rho,\eta\in QC(\Gamma)$, $P$ itself is a submersion, and
$V=P^{-1}(0)\cap R$ is an analytic submanifold of $R$ of codimension one. 
Then $-I(\rho,\eta)$ is the slope of the tangent line to 
$V\cap \{(\rho,\eta)\times\hat D\}$ at the point $(\rho,\eta,(h(\rho),0))$, so $I(\rho,\eta)$ is analytic. Theorem 
\ref{entropy analytic} then implies that $J(\rho,\eta)$ is analytic. 
\end{proof}

We obtain the following  rigidity theorem as a consequence of Theorem \ref{Manhattan}. The inequality portion of this result
was previously established by Schapira and Tapie \cite[Cor. 3.17]{schapira-tapie}.

\begin{corollary}
\label{intersection rigidity}
If $\rho,\eta\in  QC(\Gamma)$, then
$$J(\rho,\eta)\ge 1$$
with equality if and only if $\rho$ and $\eta$ are conjugate in ${\rm Isom}(\mathbb H^3)$.
\end{corollary}

\begin{proof}
Recall that the slope  $c=c(h(\rho),0)$  of $\mathcal C(\rho,\eta)$ at $(h(\rho),0)$
is given by
$$c=-\frac{\int_{\Sigma^+} \tau_\eta\ dm_{-h(\rho)\tau_\rho}}{\int_{\Sigma^+} \tau_\rho\ dm_{-h(\rho)\tau_\rho}}=-I(\rho,\eta).$$
However, by  Theorem \ref{Manhattan},
$$c\le -\frac{h(\rho)}{h(\eta)}$$
with equality if and only if $\rho$ and $\eta$ are conjugate in ${\rm Isom}(\mathbb H^3)$.
Our corollary follows immediately.
\end{proof}

\section{The pressure form}

We may define an analytic section $s:QF(S)\to QC(\Gamma)$ so that $s([\rho])$ is an element of the conjugacy
class of $\rho$. Choose co-prime hyperbolic elements $\alpha$ and $\beta$ in $\Gamma$ and let $s(\rho)$ be 
the unique element of $[\rho]$ so that $s(\rho)(\alpha)$ has attracting fixed point $0$ and repelling fixed point $\infty$
and $s(\rho)(\beta)$ has attracting fixed point $1$. This will allow us to abuse notation and regard $QF(S)$
as a subset of $QC(\Gamma)$.

Following Bridgeman \cite{bridgeman-wp} and McMullen \cite{mcmullen-pressure}, we define 
an analytic pressure form $\mathbb P$ on the tangent bundle $TQF(S)$ of $QF(S)$, by
letting 
$$\mathbb P_{T_{[\rho]}QF(S)}=s^*\Big({\rm Hess}\big(J(s(\rho),\cdot)\big)|_{T_{s(\rho)}s(QF(S))}\Big)$$
which we rewrite with our abuse of notation as:
$$\mathbb P_{T_{\rho}QF(S)}={\rm Hess}(J(\rho),\cdot))$$
Corollary \ref{intersection rigidity} implies that $\mathbb P$ is non-negative, i.e. $\mathbb P(v,v)\ge 0$ for all $v\in TQF(S)$.

Since $\mathbb P$ is non-negative, we can define a path pseudo-metric on $QF(S)$ by setting
$$d_{\mathbb P}(\rho,\eta)=\inf \left\{\int_0^1 \sqrt{\mathbb P(\gamma'(t),\gamma'(t))} dt\right\}$$
where the infimum is taken over all smooth paths in $QF(S)$ joining $\rho$ to $\eta$.

We now derive a standard criterion for when a tangent vector is degenerate with respect to $\mathbb P$, see
also \cite[Cor. 2.5]{BCS} and \cite[Lemma 9.3]{BCLS}.

\begin{lemma}
\label{degeneracy condition}
If $v\in T_\rho QF(S)$, then  $\mathbb P(v,v)=0$ if and only if
$$D_v \left( h\ell_\gamma\right) =0$$
for all $\gamma\in\Gamma$.
\end{lemma}

\begin{proof}
Let $\mathcal H_0$ denote the space of pressure zero locally H\"older continuous functions on $\Sigma^+$.
We have a well-defined Thermodynamic mapping $\psi:QF(S)\to \mathcal H_0$
given by $\psi(\rho)=-h(s(\rho))\tau_{s(\rho)}$. Notice that, by Proposition \ref{roof analytic} and
Theorem \ref{entropy analytic}, $\psi(QF(S))$ is a real analytic family.

Suppose that $\{\rho_t\}_{t\in (-\epsilon,\epsilon)}$ is an one-parameter analytic family  in $QF(S)$ and $v=\dot\rho_0$.
Then
$$\frac{d^2}{dt^2} J(\rho_0,\rho_t)\Big|_{t=0}=
\frac{d^2}{dt^2} \left(\frac{\int_{\Sigma^+} \psi(\rho_t)\ dm_{\psi(\rho_0)}}{\int_{\Sigma^+} \psi(\rho_0)\ dm_{\psi(\rho_0)}}\right)
=
\frac{\int_{\Sigma^+} \ddot\psi_0\ dm_{\psi(\rho_0)}}{\int_{\Sigma^+} \psi(\rho_0)\ dm_{\psi(\rho_0)}}$$
where
$$\ddot\psi_0= \frac{d^2}{dt^2}\Big|_{t=0} \psi(\rho_t).$$
Theorem \ref{pressure analytic} implies that
$$0=\frac{d^2}{dt^2}\Big|_{t=0}P(\psi(t))=
{\rm Var}(\dot\psi_0,m_{\psi(0)})+
\int_{\Sigma^+}\ddot\psi_0\  dm_{\psi(\rho_0)}
$$
where
$$\dot\psi_0= \frac{d}{dt}\Big|_{t=0} \psi(\rho_t),$$
so
$$\frac{d^2}{dt^2} J(\rho_0,\rho_t)\Big|_{t=0}=- \frac{{\rm Var}(\dot\psi_0,m_{\psi(0)})}{\int_{\Sigma^+} \psi(\rho_0)\ dm_{\psi(\rho_0)}}.
$$

Recall, see Sarig \cite[Thm. 5.12]{sarig-2009},
that ${\rm Var}(\dot\psi_0,m_{\psi(0)})=0$ if and only if  $\dot\psi_0$ is cohomologous to a constant function $C$.
On the other hand, since $P(\psi_t)=0$ for all $t$,
the formula for the derivative of the pressure function gives that 
$$0=\frac{d}{dt}\Big|_{t=0}P(\psi_t)=\int_{\Sigma^+}\dot\psi_0\ dm_{\psi(\rho_0)}$$
so $C$ must equal 0.
However, $\dot\psi_0$ is cohomologous to 0 if and only if for all $x\in\mathrm{Fix}^n$, and all $n$,
$$0=S_n\dot\psi_0(x)=\frac{d}{dt}\Big|_{t=0} S_n\psi_t(x)=
\frac{d}{dt}\Big|_{t=0} \Big(h(\rho_t)\ell_{G(x_1)\ldots G(x_n)}(\rho_t)\Big)
$$
(see \cite[Theorem 1.1]{sarig-2009}).
Moreover, for every hyperbolic element  $\gamma\in\Gamma$, there exists $x\in\mathrm{Fix}^n$ (for some $n$) so
that $\gamma$ is conjugate to $G(x_1)\cdots G(x_n)$, so $\ell_\gamma(\rho_t)=\ell_{G(x_1)\cdots G(x_n)}(\rho_t)$ for all $t$.
If $\gamma\in\Gamma$ is not hyperbolic, then $\ell_\gamma(\rho_t)=0$ for all $t$,
so
$$\frac{d}{dt}\Big|_{t=0} \Big(h(\rho_t)\ell_\gamma(\rho_t)\Big)=0$$
in every case.
Therefore, $\dot\psi_0$ is cohomologous to 0 if and only if 
$$\frac{d}{dt}\Big|_{t=0} \Big(h(\rho_t)\ell_{\gamma}(\rho_t)\Big)=0$$
for all $\gamma\in\Gamma$.
\end{proof}

\section{Main Theorem}

We recall that a quasifuchsian representation $\rho:\Gamma\to\mathsf{PSL}(2,\mathbb C)$ is said to be {\em fuchsian} if
it is conjugate into $\mathsf{PSL}(2,\mathbb R)$, i.e. there exists $A\in \mathsf{PSL}(2,\mathbb C)$ so that
$A\rho(\gamma)A^{-1}\in\mathsf{PSL}(2,\mathbb R)$ for all $\gamma\in\Gamma$. The Fuchsian locus $F(S)\subset QF(S)$
is the set of (conjugacy classes of) fuchsian representations.

We say that $v\in T_\rho QF(S)$ is a {\em pure bending} vector if 
$v=\frac{\partial}{\partial t}\rho_t$, $\rho=\rho_0$ is Fuchsian and $\rho_{-t}$ is the complex conjugate of $\rho_{t}$ for all $t$.
Since the Fuchsian locus $F(S)$ is the fixed point set of the action of complex conjugation on $QF(S)$
and the collection of pure bending vectors at a point in $F(S)$ is half-dimensional,
one gets a decomposition 
$$T_\rho QF(S)=T_\rho F(S)\oplus  B_\rho$$
where $B_\rho$ is the space of pure bending vectors at $\rho$.
If $v$ is a pure bending vector at $\rho\in F(S)$, then $v$ is tangent to a path obtained
by  bending $\rho$ by a (signed)  angle $t$ along
some measured lamination $\lambda$ (see Bonahon \cite[Section 2]{bonahon-almost} for details).

We are finally ready to show that our pressure form is degenerate only along pure bending vectors.

\begin{theorem}
\label{main theorem}
If $S$ is a compact hyperbolic surface with non-empty boundary, then
the pressure form $\mathbb P$ defines an ${\rm Mod}(S)$-invariant path metric $d_{\mathbb P}$ on $QF(S)$ which is
an analytic Riemannian metric except on the Fuchsian locus.

Moreover, if $v\in T_\rho(QF(S))$, then $\mathbb P(v,v)=0$ if and only if  $\rho$ is Fuchsian and $v$ is a pure bending vector.
\end{theorem}

\begin{proof}
If $v$ is a pure bending vector, then we may write $v=\dot\rho_0$ where
$\rho_{-t}$ is the complex conjugate of $\rho_t$ for all $t$, so $h\ell_\gamma(\rho_t)$ is an even function for
all $\gamma\in\Gamma$. Therefore, $D_v h\ell_\gamma=0$
for all $\gamma\in\Gamma$, so Lemma \ref{degeneracy condition} implies that $\mathbb P(v,v)=0$.

Our main work is the following converse:

\begin{proposition}
\label{only bending degenerate}
Suppose that $v\in T_\rho QF(S)$. If $\mathbb P(v,v)=0$ and $v\ne 0$, then 
$v$ is a pure bending vector.
\end{proposition}

Recall, see \cite[Lemma 13.1]{BCLS}, that if a Riemannian metric on a manifold $M$ is non-degenerate on
the complement of a submanifold $N$ of codimension at least one and the restriction of
the Riemannian metric to $TN$ is non-degenerate, then the associated path pseudo-metric
is a metric. We will see in Corollary \ref{geometricpressureform}  that the pressure metric is mapping class group invariant.
Our theorem then follows from Proposition \ref{only bending degenerate} and the fact,
established by Kao \cite{kao-pm},
that $\mathbb P$ is non-degenerate on the tangent space to the Fuchsian locus. 
\end{proof}

\medskip\noindent
{\em Proof of Proposition \ref{only bending degenerate}.}
Now suppose that $v\in T_\rho QF(S)$ and $\mathbb P(v,v)=0$.
One first observes, following Bridgeman \cite{bridgeman-wp},  that since, by Lemma \ref{degeneracy condition},
$D_v \left( h\ell_\gamma\right) =0$
for all $\gamma\in\Gamma$,
\begin{equation}
\label{ell degeneracy}
D_v\ell_\gamma=k\ell_\gamma(\rho)
\end{equation}
for all $\gamma\in\Gamma$, where  $k=-\frac{D_vh}{h(\rho)}$.

If $\gamma\in\Gamma$, then one can locally define analytic functions $tr_\gamma(\rho)$ and $\lambda_\gamma(\rho)$
which are the trace and eigenvalue of largest modulus of (some lift of) $\rho(\gamma)$. Notice that 
$\ell_\gamma(\rho)=2\log |\lambda_\gamma(\rho)|$, so we  can express
our degeneracy criterion (\ref{ell degeneracy}) as
\begin{equation}
\label{lambda degeneracy}
D_v\log|\lambda_\gamma|=k\log|\lambda_\gamma(\rho)|
\end{equation}
for all $\gamma\in\Gamma$.

We observe that Bridgeman's Lemma 7.4 \cite{bridgeman-wp} goes through nearly immediately in our setting.
We state the portion of his lemma we will need and provide a brief sketch of the proof for the reader's convenience.

\begin{lemma} {\rm (Bridgeman \cite[Lemma 7.4]{bridgeman-wp})}
\label{bridgeman calc}
If $\mathbb P(v,v)=0$, $v\in T_\rho QF(S)$,  $v\ne 0$ and $\gamma\in\Gamma$,
then $\lambda_\gamma(\rho)^2$ and $tr_\gamma(\rho)^2$ are both real.

Moreover, if $D_v tr_\alpha\ne 0$, then
$Re\left(\frac{D_v \lambda_\alpha}{\lambda_\alpha(\rho)}\right)=0$.
\end{lemma}

\begin{proof}
Suppose first that $D_vtr_\alpha\ne0$.
Since 
$$D_v(tr_\alpha)=D_v\lambda_\alpha\left(\frac{\lambda_\alpha^2-1}{\lambda_\alpha^2}\right)$$
we may conclude that $D_v\lambda_\alpha\ne0.$
Choose $\gamma\in\Gamma$, so that $\gamma$ is hyperbolic and does not commute with $\alpha$.
He then normalizes so that (the lift of)  $\rho(\alpha)=\begin{bmatrix} \lambda_\alpha & 0\\ 0& \lambda_\alpha^{-1}\\ \end{bmatrix}$ 
and (the lift of) 
$\rho(\gamma)=\begin{bmatrix} a & b\\ c& d\\ \end{bmatrix}$ where $a,b,c,d$ are all functions defined on a neighborhood of $\rho$,
such that $a$ and $d$ are non-zero. He then
computes that
$$\log |\lambda_{\alpha^n\gamma}|=n\log|\lambda_\gamma|+\log |a|+
{\rm Re}\left(\lambda_\alpha^{-2n}\left(\frac{ad-1}{a^2}\right)\right)+O(|\lambda_\alpha^{-4n}|).$$
He differentiates this equation and applies equation (\ref{lambda degeneracy}) to
conclude that 
\begin{equation}
\label{reality}
{\rm Re}\left( \frac{D_v\lambda_\alpha}{\lambda_\alpha(\rho)}\left(\frac{a(\rho)d(\rho)-1}{a(\rho)^2}\right)\right)=0.
\end{equation}
A final analysis, which breaks down into the consideration of the cases where the argument of $\lambda_\alpha^2(\rho)$ 
is rational or irrational, yields that  $\lambda_\alpha(\rho)^2$ is real. 
Since $tr^2_\alpha=\lambda_\alpha^2+2+\lambda_\alpha^{-2}$, we conclude that $tr_\alpha^2(\rho)$ is real.

One may further  differentiate the equation 
$$tr_{\alpha^n\gamma}=a\lambda_\alpha^n+d\lambda_a^{-n}$$
to conclude that
$$\lim \left(\frac{D_vtr_{\alpha^n\gamma}}{n\lambda_\alpha(\rho)^n}\right) =\frac{a(\rho)D_v\lambda_\alpha}{\lambda_\alpha(\rho)}$$
so $D_vtr_{\alpha^n\gamma}\ne 0$ is non-zero for all large enough $n$. Therefore, by the above paragraph,
$$tr_{\alpha^n\gamma}^2(\rho)=a(\rho)^2\lambda_\alpha(\rho)^{2n}+2ad(\rho)+d(\rho)^2\lambda_\alpha(\rho)^{-2n}$$ 
is real for all large enough $n$. Taking limits allows one to conclude that $a(\rho)^2$, $d(\rho)^2$ and $a(\rho)d(\rho)$ are real. 
Equation (\ref{reality})
then yields that 
$Re\left(\frac{D_v \lambda_\alpha}{\lambda_\alpha(\rho)}\right)=0$.  This completes the proof when $D_vtr_\alpha\ne 0$.

Now suppose that $D_vtr_\gamma=0$. If $\gamma$ is parabolic, $\lambda_\gamma(\rho)^2=1$ and $tr^2_\gamma(\rho)=4$
which are both real, so we may suppose that $\gamma$ is hyperbolic.
Since there are finitely many elements $\{\alpha_1,\ldots,\alpha_n\}$ of 
$\Gamma$ so that $\rho\in QF(S)$ is
determined by $\{tr_{\alpha_1}(\rho)^2,\ldots,tr_{\alpha_n}(\rho)^2\}$, see \cite[Lemma 2.5]{CCHS}, and
trace functions are analytic, there exists $\alpha\in\Gamma$, so that $D_vtr_\alpha\ne0$. 
The above analysis then yields that $a(\rho)^2$, $d(\rho)^2$ and $a(\rho)d(\rho)$ are all real. 
Therefore, 
$$tr_\gamma(\rho)^2=a(\rho)^2+2a(\rho)d(\rho)+d(\rho)^2=
\lambda_\gamma(\rho)^2+2+\lambda_\gamma(\rho)^{-2}$$
is real. So, we may conclude that $\lambda_\gamma(\rho)^2$ is real in this case as well, which completes the proof.
\end{proof}

Since $v\ne 0$, there exists  $\alpha\in\Gamma$ so that $D_vtr_\alpha\ne0$ and 
$$Re\left(\frac{D_v \lambda_\alpha}{\lambda_\alpha(\rho)}\right)=\frac{D_v |\lambda_\alpha|}{|\lambda_\alpha(\rho)|}
=D_v\log|\lambda_\alpha|,$$
equation (\ref{lambda degeneracy}) and Lemma \ref{bridgeman calc} imply that
$$k=\frac{D_v\log|\lambda_\alpha|}{\log|\lambda_\alpha(\rho)|}=0.$$
Therefore, $D_v\ell_\gamma=0$ for all $\gamma\in\Gamma$.

Notice that since $tr_\gamma(\rho)^2$ is real for all $\gamma\in\Gamma$, $\rho(\Gamma)$ lies in a proper
(real) Zariski closed subset of $\mathsf{PSL}(2,\mathbb C)$, so is not Zariski dense.
However, since the Zariski closure of $\rho(\Gamma)$ is a Lie subgroup, it must be
conjugate to a subgroup of either $\mathsf{PSL}(2,\mathbb R)$ or to the index two extension of $\mathsf{PSL}(2,\mathbb R)$
obtained by appending $z\to -z$. Since $\rho$ is quasifuchsian, its limit set $\Lambda(\rho(\Gamma))$ is a Jordan curve
and no element of $\rho(\Gamma)$ can exchange the two components of its complement. Therefore, $\rho$ is Fuchsian.
(We note that this is the only place where our argument differs significantly from Bridgeman's. It replaces his rather technical
\cite[Lemma 15]{bridgeman-wp}.)

We can then write $v=v_1+v_2$ where $v_1\in T_\rho F(S)$ and $v_2$ is a pure bending vector. 
Since $v_2$ is a pure bending vector,
$$0=D_v\ell_\gamma=D_{v_1}\ell_\gamma+D_{v_2}\ell_\gamma=D_{v_1}\ell_\gamma$$ 
for all $\gamma\in\Gamma$. But since $v_1\in T_\rho F(S)$ and there are finitely many curves whose
length functions provide analytic parameters for $F(S)$, this implies
that $v_1=0$. Therefore, $v=v_2$ is a pure bending vector.
\eproof

\section{Patterson-Sullivan measures}

In this section, we observe that the equilibrium state $m_{-h(\rho) \tau_{\rho}}$
is a  normalized pull-back of the Patterson-Sullivan measure on $\Lambda(\rho(\Gamma))$.
We use this to give a more geometric interpretation of the pressure intersection
of two quasifuchsian representations, and hence a geometric formulation of the pressure form.

Sullivan \cite{sullivan-HD,sullivan-GF} generalized Patterson's construction \cite{patterson} for
Fuchsian groups to define a  probability measure $\mu_\rho$
supported on $\Lambda(\rho(\Gamma))$,
called the {\em Patterson-Sullivan measure}. This measure satisfies the
quasi-invariance property:
\begin{equation}
d\mu(\rho(\gamma)(z))=
e^{h(\rho)B_{z}(b_0,\rho(\gamma)^{-1}(b_{0}))}d\mu_{\rho}(z)\label{eqn:pattersonsullivandefn}
\end{equation}
for all $z\in\Lambda(\rho(\Gamma))$ and $\gamma\in\Gamma$.
Sullivan showed that $\mu_{\rho}$ is a scalar multiple of
the $h(\rho)$-dimensional Hausdorff measure on $\partial\mathbb{H}^{3}$
(with respect to the metric obtained from its identification with
$T_{b_{0}}^{1}(\mathbb{H}^{3})$).

Let $\hat\mu_\rho=(\xi_\rho\circ\omega)^*\mu_\rho$ be the pull-back of the Patterson-Sullivan
measure to $\Sigma^+$.
Our normalization will involve the Gromov product with respect to $b_0$, which is defined to be
\begin{equation}
\langle z,w\rangle=\frac{1}{2}\big(B_z(b_{0},p)+B_w(b_{0},p)\big)\label{eqn:gromovproduct}
\end{equation}
for any pair $z$ and $w$ of distinct points in $\partial\mathbb{H}^{3}$, where $p$ is some (any)
point on the geodesic joining $z$ to $w$. 
One may check that for all $\alpha\in\rho(\Gamma)$ and $z,w\in\Lambda(\rho(\Gamma))$ we have
$$\langle\alpha(z),\alpha(w)\rangle =\langle z,w\rangle
-\frac{1}{2} \Big( B_{z}(b_0,\alpha^{-1}(b_0))+B_{w}(b_0,\alpha^{-1}(b_{0}))\Big).$$

If $x\in\Sigma^+$, let 
$$\Lambda(\rho(\Gamma))_x=\{\xi_\rho(\omega(y^-))\ | y\in\Sigma,\ y^+=x\},$$
where $\Sigma$ is the two-sided Markov shift associated to $\Sigma^+$ and $y^-=(y_{1-i}^{-1})_{i\in\mathbb N}$.
Notice that each $\Lambda(\rho(\Gamma))_x$ is open in $\Lambda(\rho(\Gamma))$. Furthermore, there are only
finitely many different sets which arise as $\Lambda(\rho(\Gamma))_x$ for some $x\in\Sigma^+$, since $\Lambda(\rho(\Gamma))_x$ depends only on $x_1$ and
if $r(x_1)\ge 3$ and $x_1=(b_0,w^s,w_1,\ldots,w_{k-1})$ then $\Lambda(\rho(\Gamma))_x$ depends only on $b_0$ and $w$.
Let $H_\rho:\Sigma^+\to (0,\infty)$
be defined by
$$H_{\rho}(x)=\int_{\Lambda(\rho(\Gamma))_x} e^{2h(\rho)\langle \xi_\rho(\omega(x)),z\rangle}\ d\mu_{\rho}(z).$$
Notice that $\Lambda(\rho(\Gamma))_x$ is disjoint from $\xi_\rho(I_x)$ where
$I_x$ is the component of $\partial\mathbb H^2-\partial D_0$ containing $\omega(x)$, so 
$e^{2h(\rho)\langle \xi_\rho(\omega(x)),z\rangle_{b_{0}}}$ is bounded on $\Lambda(\rho(\Gamma))_x$. In particular, $H_\rho(x)$ is
finite for all $x$.
Since $\omega$ is locally H\"older continuous and $\xi_\rho$ is H\"older, $H_\rho$ is locally H\"older continuous.

We now show that $H_\rho$ is the normalization of the pull-back $\hat\mu_\rho$ of Patterson-Sullivan measure which
gives the equilibrium measure for $-h(\rho)\tau_\rho$. Dal'bo and Peign\'e \cite[Prop. V.3]{dalbo-peigne} obtain an analogous
result for negatively curved manifolds whose fundamental groups ``act like'' geometrically finite Fuchsian groups of co-infinite area
(see also Dal'bo-Peign\'e \cite[Cor. II.5]{dalbo-peigne-ping-pong}).

\begin{proposition} \label{patterson-sullivan}
If $S$ is a compact surface with non-empty boundary and $\rho\in QF(S)$, then
the equilibrium state of  $-h(\rho)\tau_{\rho}$ on $\Sigma^+$ is a scalar multiple of
$H_\rho\ \hat\mu_\rho$.
\end{proposition}

\begin{proof}
Let $\alpha(\rho,x)=\rho(G(x_1))^{-1}$
and notice that 
$$\alpha(\rho,x)(\xi_\rho(\omega(x)))=\xi_\rho(\omega(\sigma(x)))\qquad\mathrm{and}\qquad
\alpha(\rho,x)(\Lambda(\rho(\Gamma))_x)=\Lambda(\rho(\Gamma))_{\sigma(x)}.$$
The quasi-invariance of Patterson-Sullivan measure implies that
\[
\frac{d\hat{\mu}(\sigma(y))}{d\hat{\mu}(y)}=\frac{d\mu_{\rho}\left(\alpha(\rho,x)(\xi_{\rho}(\omega(y))\right)}
{d\mu_{\rho}(\xi_{\rho}(\omega(y))}=e^{h(\rho)B_{(\xi_{\rho}(\omega)(y))}(b_{0},\alpha(\rho,x)^{-1}(b_{0}))}.
\]

We first check that $H_\rho\ \hat\mu_\rho$ is shift invariant.
{\footnotesize
\begin{eqnarray*}
H_\rho(\sigma(x)) d\hat\mu_\rho(\sigma(x)) &= &
\left(\int_{\Lambda(\rho(\Gamma))_{\sigma(x)}} e^{2h(\rho)\langle\xi_\rho(\omega(\sigma(x))),w\rangle}d\mu_\rho(w)\right) d\mu_\rho(\xi_\rho(\omega(\sigma(x)))\\
& = & 
\left(\int_{\Lambda(\rho(\Gamma))_{\sigma(x)}} e^{2h(\rho)\langle\alpha(\rho,x)(\xi_\rho(\omega(x)),\alpha(\rho,x)(v)\rangle}d\mu_\rho(\alpha(\rho,x)(v))\right) d\mu_\rho(\alpha(\rho,x)(\xi_\rho(\omega(x)))\\
& = & 
\Bigg(\int_{\Lambda(\rho(\Gamma))_x} e^{2h(\rho)\langle\xi_\rho(\omega(x)),v\rangle}e^{-h(\rho) \left( B_{\xi_\rho(\omega(x))}(b_0,\alpha(\rho,x)^{-1}(b_0)+B_{v}(b_{0},\alpha(\rho,x)^{-1}(b_0)))\right)} \\
& & \qquad e^{h(\rho)B_v(b_0,\alpha(\rho,x)^{-1}(b_0))}d\mu_\rho(v)\Bigg)
 e^{h(\rho)B_{\xi_\rho(\omega(x))}(b_0,\alpha(\rho,x)^{-1}(b_0))}d\mu_\rho(\xi_\rho(\omega(x)))\\
& = & 
\left(\int_{\Lambda(\rho(\Gamma))_x} e^{2h(\rho)<\langle\xi_\rho(\omega(x)),v\rangle} d\mu_\rho(v) \right) d\mu_\rho(\xi_\rho(\omega(x)))\\
& = & 
H_\rho(x)d\hat\mu_\rho(x)
\end{eqnarray*} }
So $H_\rho\ \hat\mu_\rho$ is shift invariant.

Now we check that $\hat\mu_\rho$ is a (scalar multiple of a) Gibbs state for $-h(\rho)\tau_\rho$. We recall, from
\cite[Theorem 2.3.3]{MU},  that it suffices to check that 
$\hat\mu_\rho$ is an eigenmeasure for  the dual of the transfer operator $\mathcal L_{-h(\rho)\tau_\rho}$.
If $g:\Sigma^{+}\to\mathbb{R}$ is bounded and continuous, then 
\begin{align*}
\int_{\Sigma^{+}}\mathcal{L}_{-h(\rho)\tau_{\rho}}(g)(x)\  d\hat{\mu}_\rho(x) & =
\int_{\Sigma^{+}}\left(\sum_{y\in\sigma^{-1}(x)}e^{-h(\rho)\tau_{\rho}(y)}g(y)\right) d\hat{\mu}_\rho(x)\\
 & =\int_{\Sigma^{+}}\left(e^{-h(\rho)\tau_{\rho}(y)}g(y)\right)\ d\hat{\mu}_\rho(\sigma(y))\\
& =\int_{\Sigma^{+}} g(y)\  d\hat{\mu}_\rho(y)\\
\end{align*}
Therefore, $\hat\mu_\rho$ is a (scalar multiple of a) Gibbs state for $-h(\rho)\tau_\rho$.

Finally, we observe that $H_\rho$ is bounded above.
If $p$  is a vertex of $D_0$, then, by construction, there exists a neighborhood $U_p$ of  $p$, so that if $\omega(x)\in U_p$,
then there exists $w\in\mathcal C^*$, so that \hbox{$x_1=(b,\omega^s,w_1,\ldots,w_{k-1},c)$} for some $s\ge 2$. 
Recall that we require that $b\ne w_{2N}$ and $c\ne  w_k$. Observe that $w_1$ is the face pairing of the edge of $D_0$ associated to $I_x$ 
and that $w_{2N}$ is the inverse of the face-pairing associated to the other edge $E$ of $\partial D_0$ which ends at $p$. 
So, if $I$ is the interval in $\partial\mathbb H^2-\partial D_0$ bounded by $E$, then $\Lambda(\rho(\Gamma))_x$  is disjoint from
$\xi_\rho(I_x\cup I)$. Therefore, $H_\rho$ is uniformly bounded on $\omega^{-1}(U_p)$ (since
$e^{2h(\rho)\langle \xi_\rho(\omega(x)),z\rangle_{b_{0}}}$ is uniformly bounded for all $z\in\Lambda(\rho(\Gamma))_x
\subset \Lambda(\rho(\Gamma))-\xi_\rho(I\cup I_x)$).
However, $D_0$ has finitely many vertices $\{p_1,\ldots,p_n\}$ and $H_\rho$ is clearly bounded above if
$\omega(x)\in \partial\mathbb H^2-\bigcup U_{p_i}$ (since
again $e^{2h(\rho)\langle \xi_\rho(\omega(x)),z\rangle_{b_{0}}}$ is uniformly bounded for all  $z\in\Lambda(\rho(\Gamma))_x
\subset \Lambda(\rho(\Gamma))-I_x$).
Therefore, $H_\rho$ is bounded above on $\Sigma^+$.

Since every multiple of a Gibbs state for $-h(\rho)\tau_\rho$ by a continuous function which is bounded between
positive constants is also a (scalar multiple of a) Gibbs state 
for $-h(\rho)\tau_\rho$ (see \cite[Remark 2.2.1]{MU}), we see
that $H_\rho\  \hat\mu_\rho$ is a shift invariant Gibbs state and hence an equilibrium measure 
for $-h(\rho)\tau_\rho$ (see Theorem \ref{uniqeq}).
\end{proof}

If $\rho\in QC(\Gamma)$, let $N_{\rho}=\mathbb{H}^{3}/\rho(\Gamma)$
be the quasifuchsian 3-manifold and let $T^{1}(N_\rho)^{nw}$ denote the
non-wandering portion of its geodesic flow. The Hopf
parameterization provides a homeomorphism
\[
\mathcal{H}:T^{1}(N_{\rho})^{nw}\to\Omega=\Big(\big(\Lambda(\rho(\Gamma))\times\Lambda(\rho(\Gamma))-\Delta\big)\times\mathbb{R}\Big)/\Gamma
\]
Let 
$$\Sigma^{\hat\tau_{\rho}}=\{(x,t):\ x\in\Sigma,\ 0\le t\leq\hat\tau_{\rho}(x^+)\}/\sim$$
(where $(x,\tau_{\rho}(x^+))\sim(\sigma(x),0)$) be the suspension flow
over $\Sigma$ with roof function $\hat\tau_{\rho}$. Recall that $\hat\tau_\rho:\Sigma^+\to (0,\infty)$ is a positive
function cohomologous to $\tau_\rho$.

The Stadlbauer-Ledrappier-Sarig coding map $\omega$ for $\Sigma^+$ extends to
a continous injective  coding map
$$\hat\omega: \Sigma\to\Lambda(\Gamma)\times\Lambda(\Gamma)$$
given by $\hat\omega(x)=(\omega(x^+),\omega(x^-))$ where
$x^+=(x_i)_{i\in\mathbb N}$ and $x^-=(x_{1-i}^{-1})_{i\in\mathbb N}$.
One then has a continuous injective map
$$\kappa:\Sigma^{\hat\tau_\rho}\to\Omega$$
which is the quotient  of the map 
$\tilde\kappa:\Sigma\times\mathbb R\to \big(\Lambda(\rho(\Gamma))\times\Lambda(\rho(\Gamma))-\Delta\big)\times\mathbb R$ 
given by 
$$\tilde\kappa(x,t)=\big( (\xi_\rho\times\xi_\rho)\hat\omega(x),t\big).$$ 
(The image of $\kappa$ is the complement of all flow lines which do not exit cusps of $N_\rho$ and has full
measure in $\Omega$.)
The map $\kappa$ conjugates the suspension flow to the geodesic flow on its image
i.e. $\kappa\circ\phi_t=\phi_t\circ \kappa$ for all $t\in\mathbb R$ on $\kappa(\Sigma^{\hat\tau_\rho})$.

The Bowen-Margulis-Sullivan measure $m_{BM}^{\rho}$ on $\Omega$ 
can be described by its lift to $\widetilde\Omega$
which is given by
\[
\widetilde{m_{BM}^{\rho}}(z,w,t)=e^{2h(\rho)\langle z,w \rangle_{b_{0}}}d\mu_{\rho}(z)d\mu_{\rho}(w)dt.
\]
The Bowen-Margulis-Sullivan measure $m_{BM}^{\rho}$ is finite and ergodic 
(see Sullivan \cite[Theorem 3]{sullivan-GF}) and equidistributed on closed
geodesics (see Roblin \cite[Th\'eor\`eme 5.1.1]{roblin} or Paulin-Pollicott-Schapira \cite[Theorem 9.11]{PPS}.)

\begin{corollary}\label{Cor:BM-EQnew} Suppose that $F:(\Sigma^+)^{\hat\tau_{\rho}}\to\mathbb{R}$
is a bounded continuous function and $\widehat F:\Sigma^{\hat\tau_\rho}\to\mathbb R$ is
given by $\widehat F(x,t)=F(x^+,t)$. Then
$$\frac{\int_{\Omega} \widehat F \circ\kappa^{-1}\ dm_{BM}^{\rho}}
{\int_{\Omega}  \ dm_{BM}^{\rho}}=\frac{\int_{\Sigma^{+}}\left(\int_{0}^{\hat\tau_\rho(x^+)}F(x,t)\ dt\right)dm_{-h(\rho)\tau_\rho}}
{\int_{\Sigma^{+}}\tau_\rho(x^+)\ dm_{-h(\rho)\hat\tau_\rho}}.$$
\end{corollary} 

\begin{proof}
Let
$$\widehat R=\{ (\hat\omega(x),t)\in \Lambda(\rho(\Gamma))\times\Lambda(\rho(\Gamma))\times\mathbb R\ |\
x\in\Sigma,\ t\in [0,\hat\tau_\rho(x^+)]\}$$
be a fundamental domain for the action of $\Gamma$ on
$\big(\Lambda(\rho(\Gamma))\times\Lambda(\rho(\Gamma))-\Delta\big)\times\mathbb R$ and let
$$ R=\{ (\omega(x^+),t)\in \Lambda(\rho(\Gamma))\times\mathbb R\ |\ x^+\in\Sigma^+,\
t\in [0,\hat\tau_\rho(x^+)]\}.$$

By Proposition \ref{patterson-sullivan}, we have
\begin{align*}
\int_{\Omega} \widehat{F}\circ\kappa^{-1} \ dm_{BM}^{\rho}
& =\int_{\widehat R}\widehat{F}\circ\kappa^{-1}e^{h(\rho)2\langle z,w \rangle_{b_{0}}}d\mu_{\rho}(z)d\mu_{\rho}(w)dt\\
 & =\int_{R} F(\omega^{-1}(z),t)\left(\int_{\Lambda(\rho(\Gamma))}e^{h(\rho)2\langle z,w\rangle_{b_{0}}}
 d\mu_\rho(w)\right)
 d\mu_{\rho}(z)dt\\
 & =\int_{R} F(\omega^{-1}(z),t))H_{\rho}(z)d\mu_{\rho}(z)dt\\
 & =\int_{\Lambda(\rho(\Gamma))}\left(\int_{0}^{\hat\tau_\rho(\omega^{-1}(z))}F(\omega^{-1}(z),t))dt \right)H_{\rho}(z)d\mu_{\rho}(z)\\
 & =\int_{\Sigma^{+}}\left(\int_{0}^{\hat\tau_{\rho}(x_{+})}F(x^{+},t)dt\right)dm_{-h(\rho)\tau_{\rho}}(x_{+})\\
\end{align*}
In particular, if we consider $F\equiv 1$, then we see that
$$||dm_{BM}^\rho||=\int_{\Omega} \ dm_{BM}^{\rho}=
\int_{\Sigma^{+}}\left(\int_{0}^{\hat\tau_{\rho}(x_{+})}dt\right)dm_{-h(\rho)\tau_{\rho}}(x_{+})
=\int_{\Sigma^{+}}\tau_\rho(x^+)\ dm_{-h(\rho)\tau_\rho}$$
so our result follows.
\end{proof}

Let 
$$\mu_T(\rho)=\frac{1}{|R_T(\rho)|}\sum_{[\gamma]\in R_T(\rho)} \frac{\delta_{[\gamma]}}{\ell_\rho(\gamma)}$$
where $\delta_{[\gamma]}$ is the Dirac measure on the closed orbit associated to $[\gamma]$ and
$$R_T(\rho)=\{[\gamma]\in[\pi_1(S)]\ |\ 0<\ell_{\rho}(\gamma)\le T\}.$$
(If $\gamma=\beta^n$ for $n>1$ and $\beta$ is indivisible, then 
$\frac{\delta_{[\gamma]}}{\ell_\rho(\gamma)}=\frac{n\delta_{[\beta]}}{\ell_\rho(\beta^n)}=\frac{\delta_{[\beta]}}{\ell_\rho(\beta)}$.)
Since the  Bowen-Margulis measure $m_{BM}^{\rho}$ is equidistributed on closed
geodesics, $\{\mu_T(\rho)\}$  converges to 
$\frac{m_{BM}^{\rho}}{||m_{BM}^{\rho}||}$
weakly (in the dual to the space of bounded continuous functions) as $T\to\infty$.

We finally obtain the promised geometric form for the pressure intersection.
We may thus think of the pressure intersection, in the spirit of Thurston, as the Hessian of the length of a random geodesic.

\begin{theorem}
\label{geometricintersection}
Suppose that $S$ is a compact surface with non-empty boundary, $X=\mathbb H^2/\Gamma$ is a finite area
surface homeomorphic to the interior of $S$ and $\rho\in QF(S)$.
If $\{\gamma_{n}\}\subset\Gamma$ and $\left\{\frac{\delta_{\rho(\gamma_{n})}}{\ell_{\rho}(\gamma_n)}\right\}$
converges weakly to $\frac{m_{BM}^{\rho}}{||m_{BM}^\rho||}$, then
\[
{\rm I}(\rho,\eta)=\lim_{n\to\infty}\frac{\ell_{\eta}(\gamma_{n})}{\ell_{\rho}(\gamma_{n})}.
\]
Moreover,
\[
{\rm I}(\rho,\eta)=\lim_{T\to\infty}\frac{1}{|R_{T}(\rho)|}\sum_{[\gamma]\in R_{T}(\rho)}\frac{\ell_{\eta}(\gamma)}{\ell_{\rho}(\gamma)}.
\]
\end{theorem}

\begin{proof}
Let $\{\Gamma_n\}$ be a sequence of finite collections of  elements of $[\Gamma]$ so that 
$\left\{\mu(\Gamma_n)=\frac{1}{|\Gamma_n|}\sum_{[\gamma]\in\Gamma_n} \frac{\delta_{[\gamma]}}{\ell_\rho(\gamma)}\right\}$ converges weakly to $\frac{m_{BM}^{\rho}}{||m_{BM}^\rho||}$.
As in \cite[Definition 3.9]{kao-pm}, consider the bounded continuous
function $\psi:\Sigma^{\hat\tau_{\rho}}\to\mathbb{R}$ given by 
\[
\psi(x,t)\longmapsto\frac{\hat\tau_{\eta}(x)}{\hat\tau_{\rho}(x)}f\left(\frac{t}{\hat\tau_{\rho}(x)}\right)\text{ for\ all}\ \ t\in[0,\hat\tau_{\rho}(x)]
\]
where $f:[0,1]\to\mathbb{R}$ is a smooth function such that $f(0)=f(1)=0,f(t)>0$
for $0<t<1$ and $\int_{0}^{1}f(t)dt=1$. Then, 
$$\int_\Omega \widehat\psi\circ\kappa^{-1}d\mu(\Gamma_n)=\frac{1}{|\Gamma_n|}\sum_{[\gamma]\in\Gamma_n}
\frac{\ell_\eta(\gamma)}{\ell_\rho(\gamma)}$$
where $\widehat \psi(x,t)=\psi(x^+,t)$ for all $x\in\Sigma$.
So, by Corollary \ref{Cor:BM-EQnew}, $\left\{\frac{1}{|\Gamma_n|} \sum_{[\gamma]\in\Gamma_n}\frac{\ell_{\eta}(\gamma_{n})}{\ell_{\rho}(\gamma_{n})}\right\} $
converges to 
\[
\frac{\int_{\Omega}\widehat\psi\circ\kappa^{-1}\ dm_{BM}^{\rho}}{||m_{BM}^{\rho}||}=
\frac{\int_{\Sigma^{+}}\frac{\hat\tau_\eta(x)}{\hat\tau_\rho(x)}\left(\int_{0}^{\hat\tau_\rho(x)}f\left(\frac{t}{\hat\tau_\rho(x)}\right)\ dt\right)dm_{-h(\rho)\tau_\rho}}
{\int_{\Sigma^{+}}\hat\tau_\rho(x)\ dm_{-h(\rho)\tau_\rho}}
=\frac{\int_{\Sigma^{+}}\hat\tau_{\eta}\ dm_{-h(\rho)\tau_{\rho}}}{\int_{\Sigma^{+}}\hat\tau_{\rho}\ dm_{-h(\rho)\tau_{\rho}}}
=\frac{\int_{\Sigma^{+}}\tau_{\eta}\ dm_{-h(\rho)\tau_{\rho}}}{\int_{\Sigma^{+}}\tau_{\rho}\ dm_{-h(\rho)\tau_{\rho}}}
\]
which completes the proof.
\end{proof}

As a consequence, we obtain a geometric presentation of the pressure form which allows us to easily see that
the pressure metric is mapping class group invariant.

\begin{corollary}
\label{geometricpressureform}
If $S$ is a compact surface with non-empty boundary and $\rho_0\in QF(S)$, then
$$\mathbb P|_{T_{\rho_0}QF(S)}=\mathrm{Hess}(J(\rho_0,\rho))=
\mathrm{Hess}\left(\frac{h(\rho)}{h(\rho_0)} \lim_{T\to\infty}\frac{1}{|R_T(\rho_0)|}\sum_{[\gamma]\in R_T(\rho_0)}\frac{\ell_\rho(\gamma)}{\ell_{\rho_0}(\gamma)}\right).$$
Moreover, the pressure metric is mapping class group invariant.
\end{corollary}

\begin{proof}
The expression for the pressure form follows immediately from the definition and Theorem \ref{geometricintersection}.
Now observe that if $\phi\in\mathrm{Mod}(S)$ and $\rho\in QF(S)$, then $\phi(\rho)=\rho\circ\phi_*$, so 
$\ell_\rho(\gamma)=\ell_{\phi(\rho)}(\phi_*(\gamma))$. Therefore, $R_T(\phi(\rho))=\phi_*(R_T(\rho))$,
so $|R_T(\rho)|=|R_T(\phi(\rho))|$ for all $T$ which implies that $h(\rho)=h(\phi(\rho))$.
We can also check that
\begin{eqnarray*}
I(\rho_0,\rho) &= & \lim_{T\to\infty}\frac{1}{|R_{T}(\rho_0)|}\sum_{[\gamma]\in R_{T}(\rho_0)}\frac{\ell_{\rho}(\gamma)}{\ell_{\rho_0}(\gamma)}\\
& =  &\lim_{T\to\infty}\frac{1}{|R_{T}(\rho_0)|}\sum_{[\gamma]\in R_{T}(\rho)}\frac{\ell_{\phi(\rho)}(\phi_*(\gamma))}{\ell_{\phi(\rho_0)}(\phi_*(\gamma))}\\
& = &  \lim_{T\to\infty}\frac{1}{|R_{T}(\phi(\rho_0))|}\sum_{[\gamma]\in R_{T}(\phi(\rho_0))}\frac{\ell_{\phi(\rho)}(\gamma)}{\ell_{\phi(\rho_0)}(\gamma)}\\
& = & I(\phi(\rho_0),\phi(\rho))\\
\end{eqnarray*}
Therefore, $J(\rho_0,\rho)=J(\phi(\rho_0),\phi(\rho))$ for all $\phi\in\mathrm{Mod}(S)$ and $\rho_0,\rho\in QF(S)$,
so the renormlized pressure intersection is mapping class group invariant, so the pressure metric is mapping class group
invariant.
\end{proof}

\end{document}